\documentclass[aos]{imsart}

\usepackage{amsmath,amssymb,amsthm}
\usepackage[numbers]{natbib}
\usepackage[dvips]{graphicx}

\theoremstyle{plain}
\newtheorem{theorem}{Theorem}[section]	
\newtheorem{lemma}{Lemma}[section]
\newtheorem{corollary}{Corollary}[section]
\newtheorem{proposition}{Proposition}[section]
\theoremstyle{definition}
\newtheorem{definition}{Definition}[section]
\newtheorem{remark}{Remark}[section]
\newtheorem{example}{Example}[section]

\DeclareMathOperator{\Card}{Card}
\DeclareMathOperator{\op}{op}
\DeclareMathOperator{\Hess}{Hess}
\DeclareMathOperator{\Tr}{Tr}

\DeclareMathOperator{\Cov}{Cov}
\DeclareMathOperator{\Var}{Var}

\newcommand{\R}{\mathbb{R}}

\newcommand{\mS}{\mathcal{S}}
\newcommand{\mF}{\mathcal{F}}

\newcommand{\mG}{\mathcal{G}}
\newcommand{\mL}{\mathcal{L}}
\renewcommand{\qed}{\hfill{\tiny \ensuremath{\blacksquare} }}%

\newcommand{\bP}{\mathbb{P}}
\newcommand{\bE}{\mathbb{E}}
\newcommand{\bEn}{\mathbb{E}_{n}}
\newcommand{\bG}{\mathbb{G}}

\begin{document}

\begin{frontmatter}

\title{Gaussian approximation of suprema of empirical processes\protect\thanksref{T1}}
\runtitle{Gaussian approximation of suprema}
\thankstext{T1}{First arXiv version:  December, 2012. Revised: April, 2014. V. Chernozhukov and D. Chetverikov  are supported by a National Science Foundation grant. K. Kato is supported by the Grant-in-Aid for Young Scientists (B) (25780152), the Japan Society for the Promotion of Science.}

\begin{aug}
\author{\fnms{Victor} \snm{Chernozhukov}\thanksref{m1}\ead[label=e1]{vchern@mit.edu}},
\author{\fnms{Denis} \snm{Chetverikov}\thanksref{m2}\ead[label=e2]{chetverikov@econ.ucla.edu}}
\and
\author{\fnms{Kengo} \snm{Kato}\thanksref{m3}
\ead[label=e3]{kkato@e.u-tokyo.ac.jp}}

\runauthor{Chernozhukov Chetverikov Kato}

\affiliation{MIT\thanksmark{m1}, UCLA\thanksmark{m2}, and University of Tokyo\thanksmark{m3}}

\address{Department of Economics and\\
Operations Research Center, MIT \\
50 Memorial Drive \\
Cambridge, MA 02142, USA.\\
\printead{e1}}

\address{Department of Economics, UCLA\\
Bunche Hall, 8283 \\
315 Portola Plaza \\
Los Angeles, CA 90095, USA.\\
\printead{e2}}

\address{Graduate School of Economics \\
University of Tokyo \\
7-3-1 Hongo, Bunkyo-ku\\
Tokyo 113-0033, Japan. \\
\printead{e3}}
\end{aug}

\begin{abstract}
This paper develops a new direct approach to approximating suprema of general empirical processes by a sequence of suprema of Gaussian processes, without taking the route of approximating whole empirical processes in the sup-norm.
We prove an abstract approximation theorem applicable to a wide variety of statistical problems, such as construction of uniform confidence bands for functions. Notably, the bound in the main approximation theorem is non-asymptotic and the theorem allow for functions that index the empirical process to be unbounded and have entropy divergent with the sample size.
The proof of the approximation theorem builds on a new coupling inequality for maxima of sums of random vectors, the proof of which depends on an effective use of Stein's method for normal approximation,  and some new empirical process techniques.
We study applications of this approximation theorem to local and series empirical processes arising in nonparametric estimation via kernel and series methods, where the classes of functions change with the sample size and are non-Donsker.
Importantly,  our new technique is able to prove the Gaussian approximation for the supremum type statistics under weak regularity conditions, especially concerning the bandwidth and the number of series functions, in those examples.
\end{abstract}



\end{frontmatter}

\section{Introduction}

\label{intro}

This paper is concerned with the problem of approximating suprema of empirical processes by a sequence of suprema of Gaussian processes.
To formulate the problem,
let $X_{1},\dots,X_{n}$ be i.i.d. random variables taking values in a measurable space $(S,\mS)$ with common distribution $P$.
Suppose that there is a sequence $\mF_{n}$ of classes  of measurable functions $S \to \R$, and consider the empirical process indexed by $\mF_{n}$:
\[
\bG_{n} f = \frac{1}{\sqrt{n}} \sum_{i=1}^{n} (f(X_{i}) - \bE[ f(X_{1}) ]), \ f \in \mF_{n}.
\]
For a moment, we implicitly assume that each $\mF_{n}$ is ``nice'' enough and postpone the measurability issue.
This paper tackles the problem of approximating $Z_{n} = \sup_{f \in \mF_{n}} \bG_{n} f$
 by a sequence of random variables $\widetilde{Z}_{n}$ equal in distribution to $\sup_{f \in \mF_{n}} B_{n} f$, where each $B_{n}$ is a centered  Gaussian process indexed by $\mF_{n}$ with covariance function $\bE [ B_{n}(f) B_{n}( g ) ] = \Cov ( f(X_{1}), g(X_{1}) )$ for all $f,g \in \mF_{n}$.
We look for conditions under which there exists a sequence of such random variables $\widetilde{Z}_{n}$ with
\begin{equation}
| Z_{n} - \widetilde{Z}_{n} | = O_{\bP}( r_{n} ), \label{problem}
\end{equation}
where $r_{n} \to 0$ as $n \to \infty$ is a sequence of constants. These results have immediate statistical implications; see Remark \ref{rem: application} and Section \ref{application} ahead.

The study of asymptotic and non-asymptotic behaviors of the supremum  of the empirical process is one of the central issues in probability theory, and dates back to the classical work of \cite{K33}.
The (tractable) distributional approximation of the supremum of the empirical process is of particular importance in mathematical statistics.  A leading example is uniform inference in nonparametric estimation, such as construction of
uniform confidence bands and specification testing in nonparametric density and regression estimation
where critical values are given by quantiles of supremum type statistics \citep[see, e.g.,][]{BR73,KP84,R94,GSV00,GN10,CCK13}.
Another interesting example appears in econometrics where there is an interest in estimating a parameter that is given as the extremum of an unknown function such as a conditional mean function. \cite{CLR12} proposed a precision-corrected estimate for such a parameter. In construction of their estimate,
approximation of quantiles of a supremum type statistic is needed, to which the Gaussian approximation  plays a crucial role.

A related but different problem is that of approximating {\em whole} empirical processes  by a sequence of Gaussian processes in the sup-norm. This problem is more difficult than (\ref{problem}).  Indeed, (\ref{problem}) is implied if there exists a sequence of versions of $B_{n}$ (which we denote by the same symbol $B_{n}$) such that
\begin{equation}
\| \bG_{n} - B_{n} \|_{\mF_{n}} := \sup_{f \in \mF_{n}} | (\bG_{n} - B_{n})f | =  O_{\bP}(r_{n}). \label{problem2}
\end{equation}
There is a large literature on the latter problem (\ref{problem2}). Notably, \citet{KMT75} (henceforth, abbreviated as KMT) proved that $\| \bG_{n} - B_{n} \|_{\mF} = O_{a.s.}(n^{-1/2} \log n)$ for $S=[0,1], P=\text{uniform distribution on $[0,1]$}$, and
$\mF= \{ 1_{[0,t]} : t \in [0,1] \}$.  See \cite{MvZ87} and \cite{BM89} for refinements of KMT's result.
\cite{M89}, \cite{K94} and \cite{R94} developed extensions of the KMT construction to more general classes of functions.

The KMT construction is a powerful tool in addressing the problem (\ref{problem2}), but when applied to general empirical processes, it typically requires strong conditions on
classes of functions and distributions.
For example,  \citet{R94} required that  $\mF_{n}$ are uniformly bounded classes of functions having uniformly bounded variations on $S = [0,1]^{d}$, and
$P$ has a continuous and positive Lebesgue density on $[0,1]^{d}$.  Such conditions are essential to the KMT construction since it depends crucially on the Haar approximation and binomial coupling inequalities of Tusn\'{a}dy. Note that \cite{K94} directly made an assumption on the accuracy of the Haar approximation of the class of functions, but still required similar side conditions to \cite{R94} in concrete applications; see Section 11 in \cite{K94}.  \cite{DP83}, \cite{BM06} and \cite{S09} considered the problem of Gaussian approximation of general empirical processes with different approaches and thereby without such side conditions.
\cite{DP83} used a finite approximation of a (possibly uncountably) infinite class of functions  and apply a coupling inequality of \cite{Y77} to the discretized empirical process
(more precisely, \cite{DP83} used a version of Yurinskii's inequality proved by \cite{D83}). \cite{BM06} and \cite{S09}, on the other hand,  used a coupling inequality of \cite{Z87} instead of Yurinskii's and some recent empirical process techniques such as Talagrand's \cite{T96} concentration inequality, which leads to refinements of Dudley and Philipp's results in some cases. However, the rates that \cite{D83}, \cite{BM06} and \cite{S09} established do not lead to tight conditions for the Gaussian approximation in non-Donsker cases, with important examples being the suprema of empirical processes arising in nonparametric estimation, namely the suprema of local and series empirical processes (see Section 3 for detailed treatment).

We develop here a new direct approach to the problem  (\ref{problem}), without taking the route of approximating the whole empirical process in the sup-norm and with different technical tools than those used in the aforementioned papers (especially the approach taken does not rely on the Haar expansion and hence differs from the KMT type approximation).
We prove an abstract approximation theorem (Theorem \ref{main}) that leads to results of type (\ref{problem}) in several situations.
The proof of the approximation theorem builds on a number of  technical tools that are of interest in their own rights: notably, 1) a new coupling inequality for maxima of sums of random vectors (Theorem \ref{Yurinskii}), where Stein's method for normal approximation (building here on \cite{CM08}  and  originally due to \cite{S72,S86}) plays an important role (see also \cite{RR09,M09,CGS11});
2) a deviation inequality for suprema of empirical processes that only requires finite moments of envelope functions (Theorem \ref{concentration}), due essentially to the recent work of \cite{BBLM05}, complemented with a new ``local'' maximal inequality for the expectation of suprema of empirical processes that extends the work of \cite{VW11} (Theorem \ref{vdVW}).
We study applications of this approximation theorem to local and series empirical processes arising in nonparametric estimation via kernel and series methods, and  demonstrate that our new technique is able to provide the Gaussian approximation for the supremum type statistics under weak regularity conditions, especially concerning the bandwidth and the number of series functions, in those examples.  A companion work \cite{CCK13} provides multiplier bootstrap methods for (approximate and valid) computation of Gaussian approximations $\widetilde Z_n$ in applications (see also Remark \ref{eq: local confidence sets} below).

It is instructive to briefly summarize here the key features of the main approximation theorem. First, the theorem establishes a non-asymptotic bound between $Z_{n}$ and its Gaussian analogue $\widetilde{Z}_{n}$. The theorem requires each $\mF_{n}$ to be pre-Gaussian (i.e., assuming the existence of  a version of $B_{n}$ that is a tight Gaussian random variable in $\ell^{\infty}(\mF_{n})$; see below for the notation), but allows for the case where the ``complexity'' of $\mF_{n}$ increases with $n$, which places the function classes outside
any fixed Donsker class;  moreover, neither the process $\bG_{n}$ nor  the supremum statistic $Z_{n}$ need to be weakly convergent as $n \to \infty$ (even after suitable normalization). Second, the bound in Theorem \ref{main} is able to exploit the ``local'' properties
of the class of functions, thereby, when applied to, say, the supremum deviation of kernel type statistics, it leads to tight conditions on the bandwidth for the Gaussian approximation (see the discussion after Theorem \ref{main} for details about these features). Note that our bound does not rely on ``smoothness'' of $\mF_{n}$ --- in contrast, in \cite{R94}, the bound on the Gaussian approximation for empirical processes depends on the total variation norm of functions. This feature is helpful in deriving good conditions on the number of series functions for the Gaussian approximation of the supremum deviation of  projection type statistics treated in Section \ref{application-2} since, for example, the total variation norm is typically large or difficult to control well for such examples.
Finally, the theorem only requires finite moments of the envelope function, which should be contrasted with \cite{K94,R94,BM06,S09} where the classes of functions studied are assumed to be uniformly bounded. Hence the theorem is readily applicable to a wide class of statistical problems to which the previous results are not, at least immediately.
We note here that although the bounds we derive are not the sharpest possible in some examples, they are better than previously available bounds in other examples, and are also of interest because of their wide applicability. In fact the results of this paper are already applied in our companion paper \cite{CCK12b} and the paper \cite{Wasserman13} by other authors.

To the best of our knowledge, \cite{NP91} is the only previous work that considered the problem of directly approximating the distribution of the supremum of the empirical process by that of the corresponding Gaussian process. However, they only cover the case where the class of functions is independent of $n$ and Donsker as the constant $C$ in their master Theorem 2 is dependent on $\mF$ (and how $C$ depends on $\mF$ is not specified), and their condition (1.4) essentially excludes the case where the ``complexity'' of $\mF$ grows with $n$, which means that their results are not applicable to the statistical problems considered in this paper (see Remark  \ref{rem: application} or Lemma \ref{lem: anticoncentration} ahead). Moreover, their approach  is significantly different from ours.

In this paper, we  substantially rely on  modern empirical process theory. For general references on empirical process theory, we refer to \cite{LT91, VW96, D99, BLM13}.  Section 9.5 of \cite{D99} has excellent historical remarks on the Gaussian approximation of empirical processes. For textbook treatments of Yurinskii's and KMT's couplings, we refer to \cite{CH93} and Chapter 10 in \cite{P02}.

\subsection{Organization}
In Section \ref{sec:main}, we present the main approximation theorem (Theorem \ref{main}). We give a proof of Theorem \ref{main} in Section \ref{proof}.  In Section \ref{application}, we study applications of Theorem \ref{main} to local and series empirical processes arising in nonparametric estimation.
Sections \ref{coupling} and \ref{empirical} are devoted to developing some technical tools needed  to prove Theorem \ref{main} and its supporting Lemma \ref{lem1}. In Section \ref{coupling}, we prove a new coupling inequality for maxima of sums of random vectors, and in Section \ref{empirical}, we present some inequalities for empirical processes.
We put some additional technical proofs, some examples, and additional results in the Appendices. Due to the page limitation, all the Appendices are placed in the Supplemental Material \cite{CCK14}.

\subsection{Notation}
\label{notation}
Let $(\Omega, \mathcal{A}, \bP)$ denote the underlying probability space. We assume that the probability space $(\Omega,\mathcal{A},\bP)$ is rich enough, in the sense that there exists a uniform random variable on $(0,1)$ defined on $(\Omega,\mathcal{A},\bP)$ independent of the sample.
For a real-valued random variable $\xi$, let $\| \xi \|_{q} = (\bE[ | \xi |^{q} ])^{1/q}, \ 1 \leq q < \infty$.
For two random variables $\xi$ and $\eta$, we write $\xi \stackrel{d}{=} \eta$
if they have the same distribution.

For any probability measure $Q$ on a measurable space $(S,\mS)$, we use the notation $Qf := \int f dQ$.
Let $\mL^{p}( Q ), p \in [1,\infty]$, denote the space of all measurable functions $f: S \to \R$ such that $\| f \|_{Q,p} := (Q | f |^{p})^{1/p} < \infty$ where $(Q | f |^{p})^{1/p}$ stands for the essential supremum when $p=\infty$.
We also use the notation $\| f \|_{\infty}:= \sup_{x \in S} | f(x) |$.
Denote by $e_{Q}$ the $\mL^{2}( Q )$-semimetric:
$
e_{Q}(f,g) = \| f - g \|_{Q,2}, \ f,g \in \mL^{2}( Q ).
$

For an arbitrary set $T$, let $\ell^{\infty}(T)$ denote the space of all bounded functions $T \to \R$, equipped with the uniform norm $\| f \|_{T} := \sup_{t \in T} | f(t) |$.
We endow $\ell^{\infty}(T)$ with the Borel $\sigma$-field induced from the norm topology.
A random variable in $\ell^{\infty}(T)$ refers to a Borel measurable map from $\Omega$ to $\ell^{\infty}(T)$.
For $\varepsilon > 0$, an {\em $\varepsilon$-net} of a semimetric space $(T,d)$ is a subset $T_{\varepsilon}$ of $T$ such that for every $t \in T$ there exists a point $t_{\varepsilon} \in T_{\varepsilon}$ with $d (t,t_{\varepsilon}) < \varepsilon$.  The {\em $\varepsilon$-covering number} $N(T,d,\varepsilon)$ of $T$ is the infimum of the cardinality of $\varepsilon$-nets of $T$, that is,  $N(T,d,\varepsilon) := \inf \{ \Card (T_{\varepsilon}) : \ \text{$T_{\varepsilon}$ is an $\varepsilon$-net of $T$} \}$ (formally define $N(T,d,0) := \lim_{\varepsilon \downarrow 0}N(T,d,\varepsilon)$, where the right limit, possibly being infinite, exists as the map $\varepsilon \mapsto N(T,d,\varepsilon)$ is non-increasing). For a subset $A$ of a semimetric space $(T,d)$, let $A^{\delta}$ denote the $\delta$-enlargement of $A$, that is,  $A^{\delta} = \{ x \in T : d(x,A) \leq \delta \}$ where $d(x,A) = \inf_{y \in A}d(x,y)$.

The standard Euclidean norm is denoted by $| \cdot |$.
The transpose of a vector $x$ is denoted by $x^{T}$.
We write $a \lesssim b$ if there exists a universal constant $C > 0$ such that $a \leq C b$.
Unless otherwise stated, $c ,C > 0$ denote  universal constants of which the values may change from place to place.
For $a,b \in \R$, we use the notation $a \vee b = \max \{ a,b \}$ and $a_{+} = a \vee 0$.

Finally, for a sequence $\{ z_{i} \}_{i=1}^{n}$, we write $\bEn [ z_{i} ] = n^{-1} \sum_{i=1}^{n} z_{i}$, that is,  $\bEn$ abbreviates the symbol  $n^{-1} \sum_{i=1}^{n}$. For example, $\bEn [ f(X_{i}) ] = n^{-1} \sum_{i=1}^{n} f(X_{i})$.

\section{Abstract approximation theorem}
\label{sec:main}

Let $X_{1},\dots,X_{n}$ be i.i.d. random variables taking values in a measurable space $(S,\mS)$ with common distribution $P$. In all what follows, we assume $n \geq 3$.
Let $\mF$ be a class of measurable functions $S \to \R$. Here we assume that the class $\mF$ is {\em $P$-centered}, that is, $Pf = 0, \ \forall f \in \mF$.
This does not lose generality since otherwise we may replace $\mF$ by $\{ f - Pf : f \in \mF \}$. Denote by $F$ a measurable {\em envelope} of $\mF$, that is,   $F$ is a non-negative measurable function $S \to \R$ such that $F(x) \geq \sup_{f \in \mF} |f (x) |, \ \forall x \in S$.

In this section the sample size $n$ is fixed, and hence the possible dependence of $\mF$ and $F$ (and other quantities) on $n$ is dropped.

We make the following assumptions.
\begin{enumerate}
\item[(A1)]
The class $\mF$ is {\em pointwise measurable}, that is,  it contains a countable subset $\mG$ such that for every $f \in \mF$ there exists a sequence $g_{m} \in \mG$ with $g_{m} (x) \to f(x)$ for every $x \in S$.
\item[(A2)] For some $q \geq 2, \ F \in \mL^{q}( P )$.
\item[(A3)] The class $\mF$ is {\em $P$-pre-Gaussian}, that is, there exists a tight Gaussian random variable $G_{P}$ in $\ell^{\infty}( \mF )$ with mean zero and  covariance function  \[
\bE [ G_{P}(f) G_{P}(g) ] = P(fg) =\bE[ f(X_{1})g(X_{1})] , \ \forall f,g \in \mF.
\]
\end{enumerate}

Assumption (A1) is made to avoid measurability complications. See Section 2.3.1 of \cite{VW96} for further discussion.
This assumption ensures that, for example, $\sup_{f \in \mF} \bG_{n} f = \sup_{f \in \mG} \bG_{n} f$, and hence the former supremum is a measurable map from $\Omega$ to $\R$. Note that by Example 1.5.10 in \cite{VW96}, assumption (A3) implies that $\mF$ is totally bounded for $e_{P}$, and $G_{P}$ has sample paths almost surely uniformly $e_{P}$-continuous.

To state the main result, we prepare some notation. For $\varepsilon > 0$, define $\mathcal{F}_{\varepsilon} = \{ f  - g : f,g \in \mF, e_{P}(f,g) < \varepsilon \| F \|_{P,2} \}$. Note that by Theorem 3.1.1 in \cite{D99}, under assumption (A3),  one can extend $G_{P}$ to the linear hull of $\mF$ in such a way that $G_{P}$ has linear sample paths (recall that the linear hull of $\mathcal{F}$ is defined as the collection of functions of the form $\sum_{j=1}^{m} \alpha_{j} f_{j}$ where $\alpha_{j} \in \R, f_{j} \in \mathcal{F}, j=1,\dots, m$). With this in mind, let
\begin{equation}\label{eq: phi of varepsilon}
\phi_{n}(\varepsilon) = \bE[ \| \bG_{n} \|_{\mF_{\varepsilon}} ] \vee \bE[ \| G_{P} \|_{\mF_{\varepsilon}} ].
\end{equation}
For the notational convenience,  let us write
\begin{equation} \label{define H}
H_{n}(\varepsilon) = \log (N(\mF, e_{P}, \varepsilon \| F \|_{P,2}) \vee n).
\end{equation}
Note that since $\mF$ is totally bounded for $e_{P}$ (because of assumption (A3)), $H_{n}(\varepsilon)$ is finite for every $0 < \varepsilon \leq 1$.
Moreover, write $M=\max_{1 \leq i \leq n} F(X_{i})$ and $\mF \cdot \mF = \{ f g : f \in \mF, g \in \mF \}$.
The following is the main theorem of this paper. The proof of the theorem will be given in Section \ref{proof}.

\begin{theorem}[\textbf{Gaussian approximation to suprema of empirical processes}]
\label{main}
Suppose that assumptions (A1), (A2) with $q \geq 3$, and (A3) are satisfied. Let
$Z = \sup_{f \in \mF} \bG_{n} f$.  Let $\kappa > 0$ be any positive constant such that $\kappa^{3} \geq \bE [ \| \bE_{n} [ |f(X_{i})|^{3} ] \|_{\mF} ]$.
Then for every $\varepsilon \in (0,1]$ and $\gamma \in (0,1)$, there exists a random variable $\widetilde{Z} \stackrel{d}{=} \sup_{f \in \mF} G_{P}f$ such that
\[
\bP \left  \{  | Z - \widetilde{Z} | >  K(q) \Delta_{n} (\varepsilon,\gamma) \right \}  \leq \gamma \left \{ 1 + \delta_n(\varepsilon,\gamma) \right \}+\frac{C \log n}{n},
\]
where $K(q) > 0$ is a constant that depends only on $q$, and
\begin{align*}
\Delta_{n} (\varepsilon,\gamma) &:= \phi_{n}(\varepsilon)  +\gamma^{-1/q} \varepsilon \| F \|_{P,2} + n^{-1/2} \gamma^{-1/q} \| M \|_{q} + n^{-1/2} \gamma^{-2/q} \| M \|_{2}  \\
&\quad  + n^{-1/4} \gamma^{-1/2} (\bE[ \| \bG_{n} \|_{\mF \cdot \mF} ])^{1/2}H_{n}^{1/2}(\varepsilon) + n^{-1/6}\gamma^{-1/3} \kappa  H_{n}^{2/3}(\varepsilon). \\
\delta_n(\varepsilon,\gamma) &:=\frac{1}{4} P\{ (F/\kappa)^{3} 1( F/\kappa > c\gamma^{-1/3} n^{1/3} H_{n}(\varepsilon)^{-1/3} ) \}.
\end{align*}
\end{theorem}


At this point, Theorem \ref{main} might seem abstract but in fact it has wide applicability. We provide a general discussion of key features of the theorem in Remark \ref{rem: key features} below after we present bounds on the main terms in the theorem. See also Corollary \ref{cor: main corollary from theorem} where we apply Theorem \ref{main} to VC type classes where many simplifications of the abstract result are possible.

Recall that we have extended $G_{P}$ to the linear hull of $\mF$ in such a way that $G_{P}$ has linear sample paths. Hence
\begin{equation*}
\| \bG_{n} \|_{\mF} = \sup_{f \in \mF \cup (-\mF)} \bG_{n} f, \ \| G_{P} \|_{\mF} =\sup_{f \in \mF \cup (-\mF)} G_{P} f,
\end{equation*}
where $-\mF := \{ -f : f \in \mF \}$, from which one can readily deduce the following corollary. Henceforth we only deal with $\sup_{f  \in \mF} \bG_{n} f$.

\begin{corollary}\label{cor: main}
The conclusion of Theorem \ref{main} continues to hold with $Z$ replaced by $Z=\| \bG_{n} \|_{\mF}$, $\widetilde{Z}$ replaced by $\widetilde{Z} \stackrel{d}{=} \| G_{P} \|_{\mF}$, and with different constants $K(q),c,C$ where $K(q)$ depends only on $q$, and $c,C$ are universal.
\end{corollary}

Theorem \ref{main} is useful only if there are suitable bounds on the following triple of terms, appearing
in its statement:
\begin{equation}
\phi_{n}(\varepsilon), \ \bE [ \| \bE_{n} [ |f(X_{i})|^{3} ] \|_{\mF} ] \ \text{and} \ \bE[ \| \bG_{n} \|_{\mF \cdot \mF} ].
\label{expectation}
\end{equation}
To bound these terms, the entropy method or the more general generic chaining method \citep{T05} are useful.
We will derive bounds on these terms using the entropy method since typically it leads to readily computable bounds. However, we leave the option of bounding the terms in (\ref{expectation}) by other means,  e.g.,  the generic chaining method (in some applications the latter is known to give sharper bounds than the entropy approach).

Consider, as in \cite[][p.239]{VW96},  the (uniform) entropy integral
\begin{equation*}
J(\delta)= J(\delta,\mF,F) = \int_{0}^{\delta} \sup_{Q} \sqrt{1+\log N(\mF, e_{Q}, \varepsilon \| F \|_{Q,2})} d \varepsilon,
\end{equation*}
where the supremum is taken over all finitely discrete probability measures on $(S,\mS)$; see \cite{VW96}, Sections 2.6 and 2.10.3, and \cite{D99}, Chapter 4, for examples where the uniform entropy integral can be suitably bounded. We assume the integral is finite:
\begin{enumerate}
\item[(A4)] $J(1,\mF,F) < \infty$.
\end{enumerate}
\begin{remark}
In applications $\mF$ and $F$ (and even $S$) may change with  $n$, that is,  $\mF= \mF_{n}$ and $F=F_{n}$. In that case, assumption (A4) is interpreted as $J(1,\mF_{n},F_{n}) < \infty$ for each $n$, but  it does allow for the case where $J(1,\mF_{n},F_{n}) \to \infty$ as $n \to \infty$. \qed
\end{remark}

We first note the following (standard) fact.

\begin{lemma}
\label{lem0}
Assumptions (A2) and (A4) imply assumption (A3).
\end{lemma}

For the sake of completeness, we verify this lemma in  the Supplemental Material \cite{CCK14}.
The following lemma provides bounds on the quantities in (\ref{expectation}). Its proof is given in the Supplemental Material \cite{CCK14}.


\begin{lemma}[Entropy-based bounds on the triple (\ref{expectation})]
\label{lem1}
Suppose that assumptions (A1), (A2) and (A4) are satisfied.
Then for $\varepsilon \in (0,1]$,
\[
\phi_{n}(\varepsilon) \lesssim  J(\varepsilon) \| F \|_{P,2}  +  n^{-1/2} \varepsilon^{-2} J^{2}(\varepsilon) \| M \|_{2}.
\]
Moreover, suppose that assumption (A2) is satisfied with $q \geq 4$, and for $k=3,4$, let $\delta_{k} \in (0,1]$ be any positive constant such that $\delta_{k} \geq \sup_{f \in \mF} \| f \|_{P,k}/\| F \|_{P,k}$. Then
\begin{align*}
&\bE [ \| \bE_{n} [ |f(X_{i})|^{3} ] \|_{\mF} ] - \sup_{f \in \mF} P| f |^{3} \\
&\qquad \lesssim n^{-1/2} \| M \|_{3}^{3/2} \left [ J(\delta_{3}^{3/2},\mF,F) \| F \|^{3/2}_{P,3}  + \frac{\| M \|^{3/2}_{3} J^{2}(\delta_{3}^{3/2},\mF,F)}{\sqrt{n} \delta_{3}^{3}} \right ], \\
&\bE [ \| \bG_{n} \|_{\mF \cdot \mF} ] \lesssim J(\delta_{4}^{2},\mF,F) \| F \|_{P,4}^{2} +  \frac{\| M \|_{4}^{2} J^{2}(\delta^{2}_{4},\mF,F)}{\sqrt{n} \delta_{4}^{4}}.
\end{align*}
\end{lemma}

\begin{remark}[On the usefulness of the above bounds]
The bounds above are designed to handle cases when the suprema of weak moments, $P| f |^{3}$ and $Pf^{4}$, are much smaller than the moments of the envelope function, which is the case for all the examples studied in Section \ref{application} where all the proofs for the results in that section follow from application of Corollary \ref{cor: main corollary from theorem} below, which is a direct consequence of Theorem \ref{main} and Lemma \ref{lem1}.
\qed
\end{remark}

\begin{remark}[\textbf{Key features of Theorem \ref{main}}]\label{rem: key features}
Before going to the applications, we discuss the key features of Theorem \ref{main}.
First, Theorem \ref{main} does not require uniform boundedness of $\mF$, and requires only finite moments of the envelope function.
This should be contrasted with the fact that many papers working on the Gaussian approximation of empirical processes in the sup-norm, such as \cite{K94,R94,BM06,S09}, required that  classes of functions are uniformly bounded.
There are, however, many statistical applications where uniform boundedness of the class of functions is too restrictive,  and the generality of Theorem \ref{main} in this direction will turn out to be useful --- a typical example of such an application is the problem of performing inference on a nonparametric regression function with unbounded noise using kernel and series estimation methods.
One drawback  is that  $\gamma$, which in applications we take as $\gamma = \gamma_{n} \to 0$, is typically at most $O(n^{-1/2})$, and hence Theorem \ref{main} generally gives
only ``in probability bounds'' rather than ``almost sure bounds'' (though in some cases, it is possible to derive ``almost sure bounds'' from this theorem; see, in particular, Appendix C of the Supplemental Material).
The second feature of Theorem \ref{main} is that it is able to exploit the ``local'' properties of the class of functions $\mF$. By Lemma \ref{lem1}, typically, we may take $\kappa^{3} \approx \sup_{f \in \mF} P | f |^{3}$ and
$\bE [ \| \bG_{n} \|_{\mF \cdot \mF} ] \approx \sup_{f \in \mF} \sqrt{P f^{4}}$ (up to logarithmic in $n$ factors). In some applications, for example, nonparametric kernel and series estimations considered in the next section, the class $\mF=\mF_{n}$ changes with $n$ and $\sup_{f \in \mF_{n}} \| f \|_{P,k}/\| F_{n} \|_{P,k}$ with $k=3,4$ decrease to $0$ where $F_{n}$ is an envelope function of $\mF_{n}$. The bound in Theorem \ref{main} (with help of Lemma \ref{lem1}) effectively exploits this information and leads to tight conditions on, say, the bandwidth and the number of series functions for the Gaussian approximation; roughly the theorem gives bounds on the approximation error of the form $(n h_n^d)^{-1/6}$ for kernel estimation and $(K_n/n)^{-1/6}$ for series estimation (up to logarithmic in $n$ factors),  where $h_n \to 0$ is the bandwidth and $K_n \to \infty$ is the number of series functions. This feature will  be clear from the proofs for the applications in the following section. \qed
\end{remark}

\begin{remark}[An application to VC type classes]
Although applications of the general results in this section are not restricted to VC type classes, combination of Theorem \ref{main} and Lemma \ref{lem1} will lead to a simple bound for these classes.
Recall the definition of VC type classes:
\begin{definition}[VC type class]
Let $\mF$ be a class of measurable functions on a measurable space $(S,\mS)$, to which a measurable envelope $F$ is attached.
We say that $\mF$ is {\em VC type} with envelope $F$  if there are constants $A,v > 0$ such that
$\sup_{Q} N(\mF,e_{Q},\varepsilon \| F \|_{Q,2}) \leq (A/\varepsilon)^{v}$ for all $0 < \varepsilon \leq 1$, where the supremum is taken over all finitely discrete probability measures on $(S,\mS)$.
\end{definition}
Note that the definition of VC type classes allows for unbounded envelops $F$.  The VC type class is a wider concept than VC {\em subgraph} class (\cite{VW96}, Chapter 2.6). The VC type property is ``stable'' under summation, product, or
more generally Lipschitz-type transformations,  making it much easier to check whether a function class
is VC type; see Lemma \ref{lem: uniform cov} in the Supplemental Material \cite{CCK14}.

We have the following corollary of Theorem \ref{main}, whose proof is given in the Supplemental Material \cite{CCK14}.
\begin{corollary}[\textbf{Gaussian approximation to suprema of empirical processes indexed by VC type classes}]\label{cor: main corollary from theorem}
Suppose that assumption (A1) is satisfied. In addition, suppose that the class $\mathcal{F}$ is VC type with an envelope $F$ and constants $A\geq e$ and $v\geq 1$. Suppose also that for some $b\geq\sigma>0$ and $q\in[4,\infty]$, we have $\sup_{f\in\mathcal{F}}P|f|^k\leq \sigma^2b^{k-2}$ for $k=2,3,4$ and $\|F\|_{P,q}\leq b$. Let $Z=\sup_{f\in\mathcal{F}}\mathbb{G}_n f$. Then for every $\gamma\in(0,1)$, there exists a random variable $\widetilde{Z}\stackrel{d}{=} \sup_{f \in \mF} G_{P}f$ such that
\begin{align*}
&\mathbb{P}\left\{|Z-\widetilde{Z}|>\frac{b K_n}{\gamma^{1/2} n^{1/2-1/q}}+\frac{(b \sigma)^{1/2}K_n^{3/4}}{\gamma^{1/2}n^{1/4}}+\frac{(b \sigma^2 K_n^2)^{1/3}}{\gamma^{1/3}n^{1/6}}\right\}\\
&\qquad \leq C\left(\gamma+\frac{\log n}{n}\right),
\end{align*}
where $K_n=c v(\log n\vee\log(A b/\sigma))$, and $c,C$ are constants that depend only on $q$ (``$1/q$'' is interpreted as ``$0$'' when $q=\infty$).
\end{corollary}
\qed
\end{remark}

\begin{remark}[Gaussian approximation in the Kolmogorov distance]
\label{rem: application}
Theorem \ref{main} combined with Lemma \ref{lem1} can be used to show that the result (\ref{problem}) holds for some sequence of constants $r_n \to 0$ (subject to some conditions; possible rates of $r_{n}$ are problem-specific). In statistical applications, however, one is typically interested in the result of the form (here we follow the notation used in Section \ref{intro})
\begin{equation}
\sup_{t \in \R} | \bP ( Z_{n} \leq t ) - \bP ( \widetilde{Z}_{n} \leq t ) | = o(1), \ n \to \infty. \label{eq: statistical applications}
\end{equation}
That is, the approximation of the distribution of $Z_{n}$ by that of $\widetilde{Z}_{n}$ in the Kolmogorov distance is required.
To derive (\ref{eq: statistical applications}) from (\ref{problem}), we invoke the following lemma.

\begin{lemma}[Gaussian approximation in Kolmogorov distance: non-asymptotic result]
\label{lem: Kolmogorov}
Consider the setting described in the beginning of this section.
Suppose that assumptions (A1)-(A3) are satisfied, and that there exist constants  $\underline{\sigma}, \bar{\sigma}>0$ such that $\underline{\sigma}^{2} \leq Pf^{2} \leq \bar{\sigma}^{2}$ for all $f\in\mathcal{F}$. Moreover, suppose that there exist constants $r_{1},r_{2} > 0$ and a random variable $\widetilde{Z} \stackrel{d}{=} \sup_{f \in \mF} G_{P}f$ such that $\bP  \{  | Z - \widetilde{Z} | > r_{1} \}  \leq r_{2}$.
Then
\[
\sup_{t \in \R} | \bP ( Z \leq t ) - \bP( \widetilde{Z} \leq t ) | \leq C_{\sigma}r_{1} \left\{ \bE[ \widetilde{Z} ] +\sqrt{1\vee \log(\underline{\sigma}/r_{1})} \right \}  + r_{2},
\]
 where $C_{\sigma}$ is a constant depending only on $\underline{\sigma}$ and $\bar{\sigma}$.
\end{lemma}

It is now not difficult to give conditions to deduce (\ref{eq: statistical applications}) from (\ref{problem}). Formally, we state the following lemma.

\begin{lemma}[Gaussian approximation in Kolmogorov distance: asymptotic result]
\label{lem: Kolmogorov2}
Suppose that there exists a sequence of ($P$-centered) classes  $\mF_{n}$ of measurable functions $S \to \R$ satisfying assumptions (A1)-(A3) with $\mF=\mF_{n}$ for each $n$, and that there exist constants $\underline{\sigma}, \bar{\sigma}>0$ (independent of $n$) such that $\underline{\sigma}^{2} \leq Pf^{2} \leq \bar{\sigma}^{2}$ for all $f \in \mF_{n}$. Let $Z_{n} = \sup_{f \in \mF_{n}} \bG_{n}f$, and denote by $B_{n}$ a tight Gaussian random variable in $\ell^{\infty}(\mF_{n})$ with mean zero and  covariance function $\bE [ B_{n}(f) B_{n}( g ) ] = P(fg)$ for all $f,g \in \mF_{n}$. Moreover, suppose that there exist a sequence of random variables $\widetilde{Z}_{n} \stackrel{d}{=} \sup_{f \in \mF_{n}} B_{n}f$ and a sequence of constants $r_{n} \to 0$ such that $| Z_{n} - \widetilde{Z}_{n} | = O_{\bP}(r_{n})$ and $r_{n} \bE[ \widetilde{Z}_{n} ] = o(1)$ as $n \to \infty$. Then as $n \to \infty$,
$\sup_{t \in \R} | \bP ( Z_{n} \leq t ) - \bP ( \widetilde{Z}_{n} \leq t ) | = o(1)$.
\end{lemma}
Note here that we allow the case where $\bE[ \widetilde{Z}_{n} ] \to \infty$. In the examples handled in the following section, typically, we have $\bE [ \widetilde{Z}_{n} ] = O(\sqrt{\log n})$.  We note that the companion work \cite{CCK13} provides multiplier bootstrap methods for uniformly consistent estimation of the map $t \mapsto \bP ( \widetilde{Z}_{n} \leq t)$ in applications (see also Remark \ref{eq: local confidence sets} below). \qed \end{remark}

\section{Applications}

\label{application}

This section studies applications of Theorem \ref{main} and its supporting Lemma \ref{lem1} (via Corollary \ref{cor: main corollary from theorem}) to local and series empirical processes arising in nonparametric estimation via kernel and series methods. In both examples, the classes of functions change with the sample size $n$ and the corresponding processes $\bG_{n}$ do not have tight limits. Hence regularity conditions for the Gaussian approximation for the suprema will be of interest. 
All the proofs in this section, and motivating examples for series empirical processes treated in Section \ref{application-2}, are gathered in the Supplemental Material \cite{CCK14}.

\subsection{Local empirical processes}

\label{application-1}

This section applies Theorem \ref{main} to the supremum deviation of kernel type statistics.
Let $(Y_{1},X_{1}),\dots,(Y_{n}, X_{n})$ be i.i.d. random variables taking values in the product space $\mathcal{Y} \times \R^{d}$, where $(\mathcal{Y},\mathcal{A}_{\mathcal{Y}})$ is an arbitrary measurable space.
Suppose that there is a class $\mG$ of measurable functions $\mathcal{Y} \to \R$. Let $k(\cdot)$ be a kernel function  on $\R^{d}$.
By ``kernel function'', we simply mean that $k(\cdot)$ is integrable with respect to the Lebesgue measure on $\R^{d}$ and its integral on $\R^{d}$ is normalized to be $1$, but we do not assume $k(\cdot)$ to be non-negative, that is,  higher order kernels are allowed.
Let $h_{n}$ be a sequence of positive constants such that $h_{n} \to 0$ as $n \to \infty$, and let $\mathcal{I}$ be an arbitrary Borel subset of $\R^{d}$.
Consider the kernel-type statistics
\begin{equation}\label{eq: kernel-type statistics}
S_{n}(x,g) = \frac{1}{nh_{n}^{d}} \sum_{i=1}^{n} g(Y_{i}) k(h_{n}^{-1} (X_{i}-x)), \ (x,g) \in \mathcal{I} \times \mG.
\end{equation}
Typically, under suitable regularity conditions, $S_{n}(x,g)$  will be a consistent estimator of $\bE[ g(Y_{1}) \mid X_{1}=x] p(x)$, where $p(\cdot)$ denotes a Lebesgue density of the distribution of $X_{1}$ (assuming its existence). For example, when $g \equiv 1$, $S_{n}(x,g)$ will be a consistent estimator of $p(x)$;  when $\mathcal{Y}=\R$ and $g(y) = y$, $S_{n}(x,g)$ will be a consistent estimator of $\bE [ Y_{1} \mid X_{1} = x] p(x)$; and when $\mathcal{Y}=\R$ and $g(\cdot) = 1(\cdot \leq y), y \in \R$, $S_{n}(x,g)$ will be a consistent estimator of $\bP (Y_{1} \leq y \mid X_{1}=x) p(x)$.
In statistical applications, it is often of interest to approximate the distribution of the following quantity:
\begin{equation}\label{eq: suprema of kernel-type statistics}
W_{n} = \sup_{(x,g) \in \mathcal{I} \times \mG} c_{n}(x,g) \sqrt{nh_{n}^{d}} (S_{n}(x,g) - \bE [ S_{n}(x,g) ] ),
\end{equation}
where $c_{n}(x,g)$ is a suitable normalizing constant. A typical choice of $c_{n}(x,g)$ would be such that $\Var (\sqrt{nh_{n}^{d}} S_{n}(x,g)) = c_{n}(x,g)^{-2}  + o(1)$.
Limit theorems for $W_{n}$ are developed in \cite{BR73,KP84,DM94,R94,EM97,M04}, among others.

\cite{EM97} called the process $g \mapsto \sqrt{nh_{n}^{d}} (S_{n}(x,g) - \bE [ S_{n}(x,g) ] )$ a ``local'' empirical process at $x$ (the original definition of the local empirical process in \cite{EM97} is slightly more general in that $h_{n}$ is replaced by a
sequence of bi-measurable functions). With a slight abuse of terminology, we also call the process $(x,g) \mapsto \sqrt{nh_{n}^{d}} (S_{n}(x,g) - \bE [ S_{n}(x,g) ] )$ a local empirical process.

We consider the problem of approximating $W_{n}$ by a sequence of suprema of Gaussian processes.
For each $n \geq 1$,  let $B_{n}$ be a centered Gaussian process indexed by $\mathcal{I} \times \mG$ with covariance function
\begin{multline}
\bE [ B_{n}(x,g) B_{n}(\check{x},\check{g}) ] \\
=  h_{n}^{-d} c_{n}(x,g) c_{n}(\check{x},\check{g}) \Cov[ g(Y_{1}) k(h_{n}^{-1} (X_{1}-x)), \check{g}(Y_{1}) k(h_{n}^{-1} (X_{1}-\check{x}))]. \label{covfunc}
\end{multline}
It is expected that under suitable regularity conditions, there is a sequence $\widetilde{W}_{n}$ of random variables such that
$\widetilde{W}_{n} \stackrel{d}{=} \sup_{(x,g) \in \mathcal{I} \times \mG} B_{n}(x,g)$ and as $n \to \infty$, $| W_{n} - \widetilde{W}_{n} | \stackrel{\bP}{\to} 0$.
We shall argue the validity of this approximation with explicit rates.


We make the following assumptions.
\begin{enumerate}
\item[(B1)] $\mG$ is a pointwise measurable class of functions $\mathcal{Y} \to \R$ uniformly bounded by a constant $b > 0$, and is  VC type with envelope $\equiv b$.
\item[(B2)]  $k(\cdot)$ is a  bounded and continuous kernel function on $\R^{d}$, and such that the class of functions $\mathcal{K} = \{ t \mapsto k(h t + x ) : h > 0, x \in \R^{d} \}$ is VC type with envelope $\equiv \| k \|_{\infty}$.
\item[(B3)] The distribution of $X_{1}$ has a bounded Lebesgue density $p(\cdot)$ on $\R^{d}$.
\item[(B4)] $h_{n} \to 0$ and $\log(1/h_{n})= O(\log n)$ as $n \to \infty$.
\item[(B5)] $C_{\mathcal{I} \times \mG} := \sup_{n \geq 1} \sup_{(x,g) \in \mathcal{I} \times \mG} | c_{n}(x,g) | < \infty$.
Moreover, for every fixed $n \geq 1$ and for every $(x_{m},g_{m}) \in \mathcal{I} \times \mG$ with $x_{m} \to x \in \mathcal{I}$ and $g_{m} \to g \in \mG$  pointwise,
$c_{n}(x_{m},g_{m}) \to c_{n}(x,g)$.
\end{enumerate}
We note that \cite{NP87} and especially \citep[]{GG02,GN09} give general sufficient conditions under which $\mathcal{K}$ is VC type.

We first assume that $\mG$ is uniformly bounded, which will be relaxed later.

\begin{proposition}[Gaussian approximation to suprema of local empirical processes: bounded case]
\label{local:uniform}
Suppose that assumptions (B1)-(B5) are satisfied.
Then for every $n \geq 1$, there is a tight Gaussian random variable $B_{n}$ in $\ell^{\infty}(\mathcal{I} \times \mG)$ with mean zero  and covariance function (\ref{covfunc}), and there is a sequence $\widetilde{W}_{n}$ of random variables such that $\widetilde{W}_{n} \stackrel{d}{=} \sup_{(x,g) \in \mathcal{I} \times \mG} B_{n}(x,g)$ and as $n \to \infty$,
\begin{equation*}
| W_{n} - \widetilde{W}_{n} |  = O_{\bP}\{ (nh_{n}^{d})^{-1/6} \log n + (nh_{n}^{d})^{-1/4} \log^{5/4} n + (nh_{n}^{d})^{-1/2} \log^{3/2} n \}.
\end{equation*}
\end{proposition}

Even when $\mG$ is not uniformly bounded, a version of Proposition \ref{local:uniform}  continues to hold provided that suitable restrictions on the moments of the envelope of $\mG$ are assumed. Instead of assumption (B1), we make the following assumption.

\begin{enumerate}
\item[(B1)$'$] $\mG$ is a pointwise measurable class of functions $\mathcal{Y} \to \R$ with measurable envelope $G$ such that $\bE[ G^{q}(Y_{1}) ] < \infty$ for some $q \geq 4$
and $\sup_{x \in \R^{d}} \bE[ G^{4}(Y_{1}) \mid X_{1} = x ] < \infty$. Moreover,  $\mG$ is VC type with envelope $G$.
\end{enumerate}

Then we have the following proposition.
\begin{proposition}[Gaussian approximation to suprema of local empirical processes: unbounded case]
\label{local:moment}
Suppose that assumptions (B1)$'$ and (B2)-(B5) are satisfied.
Then the conclusion of Proposition \ref{local:uniform} continues to hold, except for that the speed of approximation is
\begin{equation*}
O_{\bP}\{ (nh_{n}^{d})^{-1/6} \log n  + (nh_{n}^{d})^{-1/4} \log^{5/4} n + (n^{1-2/q}h_{n}^{d})^{-1/2} \log^{3/2} n \}.
\end{equation*}
\end{proposition}

\begin{remark}[Discussion and comparison to other results]
\label{com_Rio}
It is instructive to compare Propositions \ref{local:uniform} and \ref{local:moment} with implications of Theorem 1.1 of \citet{R94}, which is a very sharp result on the Gaussian approximation (in the sup-norm)  of general empirical processes indexed by uniformly bounded VC type classes of functions having  locally uniformly bounded variation.

1.  Rio's \cite{R94} Theorem 1.1 is not applicable to the case where the envelope function $G$ is not bounded. Hence Proposition \ref{local:moment} is not covered by \cite{R94}.
Indeed, we are not aware of any previous result that leads to the conclusion of Proposition \ref{local:moment}, at least in this generality.
For example, \cite{KP84} considered the Gaussian approximation of $W_{n}$ in the case where $\mathcal{Y}=\R$ and $g(y) = y$, but also assumed that the support of $Y_{1}$ is bounded.
\cite{EM97} proved in their Theorem 1.1 a weak convergence result for local empirical processes, which, combined with the Skorohod representation and Lemma \ref{strassen} ahead, implies a Gaussian approximation result for $W_{n}$ even when $\mG$ is not uniformly bounded (but without explicit rates); however, their  Theorem 1.1 (and also Theorem 1.2) is tied with the single value of $x$, that is,  $x$ is fixed, since both theorems assume that the ``localized'' probability measure, localized at a given $x$, converges (in a suitable sense) to a fixed probability measure (see assumption (F.ii) in \cite{EM97}).
The same comment applies to \cite{EM98}. In contrast, our results apply to the case where the supremum is taken over an uncountable set of values of $x$, which is relevant to statistical applications such as construction of uniform confidence bands.

2. In the special case of kernel density estimation (i.e., $g \equiv 1$), Rio's Theorem 1.1 implies (subject to some regularity conditions) that $| W_{n} - \widetilde{W}_{n} | =O_{a.s.} \{ (nh_{n}^{d})^{-1/(2d)} \sqrt{\log n} + (n h_{n}^{d})^{-1/2} \log n \}$ for $d \geq 2$ (the $d=1$ case is formally excluded from \cite{R94} but Gin\'{e} and Nickl showed that the same bound can be obtained for $d=1$ case \citep[the proof of Proposition 5 in][]{GN10}). Hence Rio-Gin\'{e}-Nickl's error rates are better than ours when $d =1,2,3$, but ours are better when $d \geq 4$ (aside from the difference between ``in probability'' and almost sure bounds). Another approach to couplings of kernel density estimators is proposed in Neumann \cite{N98} where the distribution of $W_n$ is coupled to the distribution of the smoothed bootstrap, which is then coupled to the distribution of the empirical bootstrap. Neumann's Theorem 3.2 implies that one can construct a sequence $X_1,\dots,X_n$, its copy $\overline{X}_1,\dots,\overline{X}_n$, and empirical bootstrap sample $X_1^*,\dots,X_n^*$ from $\overline{X}_1,\dots,\overline{X}_n$ so that if we define $W_n^*$ by (\ref{eq: kernel-type statistics}) and (\ref{eq: suprema of kernel-type statistics}) with $X_1,\dots,X_n$ replaced by $X_1^*,\dots,X_n^*$, then $|W_n-W_n^*|=O_{\bP}((n h_d)^{-1/(2+d)}(\log n)^{(4+d)/(2(2+d))})$. Thus Neumann's error rates of (empirical bootstrap) approximation are better than our error rates of (Gaussian) approximation when $d\leq 4$ but ours are better when $d\geq 5$. Also we note that Neumann's approach requires similar side conditions as those of Rio's approach, is tied with kernel density estimation and not as general as ours.

3. Consider, as a second example, kernel regression estimation (that is, $\mathcal{Y}=\R$ and $g(y) = y$). In order to formally apply Rio's Theorem 1.1 to this example, we need to assume that, for example, $(Y_{1},X_{1})$ is generated in such a way that
$(Y_{1},X_{1})=(h(U,X_{1}),X_{1})$ where the joint distribution of $(U,X_{1})$ has support $[0,1]^{d+1}$ with continuous and positive Lebesgue density on $[0,1]^{d+1}$, and $h$ is a function $[0,1]^{d+1} \to \R$ which is bounded and of bounded variation [for example, let $F_{Y_{1} \mid X_{1}}^{-1} ( \cdot \mid x )$ denote the quantile function of the conditional distribution of $Y_{1}$ given $X_{1}=x$ and take $U$ uniformly distributed on $(0,1)$ independent of $X_{1}$; then $(Y_{1},X_{1}) \stackrel{d}{=} (F^{-1}_{Y_{1} \mid X_{1}}(U \mid X_{1}), X_{1})$, but for the above condition to be met, we need to assume that $F_{Y_{1} \mid X_{1}}^{-1} ( u \mid x )$ is (bounded and) of bounded variation as a function of $u$ {\em and}  $x$, which is not a typical assumption in estimation of the conditional mean]. Subject to such side conditions, Rio's Theorem 1.1 leads to the following error rate: $| W_{n} - \widetilde{W}_{n} | =O_{a.s.} \{ (n^{d/ (d+1)} h_{n}^{d})^{-1/(2d)} \sqrt{\log n} +  (n h_{n}^{d})^{-1/2} \log n \}$. See, for example, \cite{CLR12}, Theorem 8.
In contrast, Propositions \ref{local:uniform} and \ref{local:moment} do not require such side conditions.
Moreover, aside from the difference between ``in probability'' and almost sure bounds,  as long as $h_{n} = O(n^{-a})$ for some $a > 0$, our error rates are always better when $d \geq 2$.
When $d = 1$, our rate is better as long as $nh_{n}^{4}/\log^{c} n \to 0$ (and vice versa) where $c > 0$ is some constant.\qed
\end{remark}

%
\begin{remark}[Converting coupling to convergence in Kolmogorov distance]\label{rem: local processes Kolmogorov}
By Remark \ref{rem: application}, we can convert the results in Propositions \ref{local:uniform} and \ref{local:moment} into convergence of the Kolmogorov distance between the distributions of $W_{n}$ and its Gaussian analogue $\widetilde{W}_{n}$. In fact, under either the assumptions of Proposition \ref{local:uniform} or \ref{local:moment}, by Dudley's inequality for Gaussian processes \citep[][Corollary 2.2.8]{VW96}, it is not difficult to deduce that
$\bE[ \widetilde{W}_{n}] = O(\sqrt{\log n})$. Hence if moreover there exists a constant $\underline{\sigma} > 0$ (independent of $n$) such that  $ \Var (c_{n}(x,g)\sqrt{nh_{n}^{d}} S_{n}(x,g)) \geq \underline{\sigma}^{2}$ for all $(x,g) \in \mathcal{I} \times \mathcal{G}$ (giving primitive regularity conditions for this assumption is a standard task; note also that under either the assumptions of Proposition \ref{local:uniform} or \ref{local:moment}, $\Var (c_{n}(x,g)\sqrt{nh_{n}^{d}} S_{n}(x,g))$ is bounded from above uniformly in $(x,g) \in \mathcal{I} \times \mG$), we have
\[
| W_{n} - \widetilde{W}_{n} | = o_{\bP}(\log^{-1/2} n) \Rightarrow \sup_{t \in \R} | \bP( W_{n} \leq t) - \bP( \widetilde{W}_{n} \leq t) | = o(1).
\]
Note that $| W_{n} - \widetilde{W}_{n} | = o_{\bP}(\log^{-1/2} n)$ (i) if $nh_{n}^{d} /\log^{c} n \to \infty$ under the assumptions of Proposition \ref{local:uniform}, and (ii) if $n^{(1-2/q)} h_{n}^{d}/ \log^{c} n \to \infty$ under the assumptions of Proposition \ref{local:moment}, where $c >0$ is some constant. These conditions on the bandwidth $h_{n}$ are mild, and interestingly they essentially coincide with the conditions on the bandwidth used in establishing exact rates of uniform strong consistency of kernel type estimators in \cite{EM00,EM05}. \qed
\end{remark}

\begin{remark}[Constructing under-smoothed uniform bands]\label{eq: local confidence sets}
The results in Propositions \ref{local:uniform} and \ref{local:moment} are useful for constructing one- and two-sided uniform confidence bands for various nonparametric functions, such as density and conditional mean, estimated via kernel methods. For concreteness, consider a kernel density estimator $\widehat{S}_n(x)=S_n(x,g)$ defined in (\ref{eq: kernel-type statistics}) with $g\equiv 1$. Let $\sigma_n(x)=\sqrt{\Var(\widehat{S}_n(x))}$, and define $W_n$ as in (\ref{eq: suprema of kernel-type statistics}) with $c_n(x,g)=1/(\sigma_n(x)\sqrt{n h_n^d})$. Also define $\mathcal{C}_n(x)=[\widehat{S}_n(x)-c(\alpha)\sigma_n(x),\infty)$ where $c(\alpha)$ is a constant specified later with $\alpha \in (0,1)$ a confidence level. Assume that the bandwidth $h_n$ is chosen in such a way that
\begin{equation}\label{eq: undersmoothing}
\sup_{x\in\mathcal{I}}\frac{|\bE[\widehat{S}_n(x)]-p(x)|}{\sigma_n(x)}=o(\log^{-1/2} n).
\end{equation}
Conditions like (\ref{eq: undersmoothing}) are typically referred to as under-smoothing \citep[see][p.1130 for related discussion]{GN10}. Then
\begin{align}
&\bP(p(x)\in\mathcal{C}_n(x),\forall x\in\mathcal{I}) \leq \bP(W_n\leq c(\alpha) + o(\log^{-1/2} n) ) \notag \\
&\quad =\bP(\widetilde{W}_n\leq c(\alpha) + o(\log^{-1/2} n) ) + o(1) = \bP(\widetilde{W}_n\leq c(\alpha)) + o(1), \label{eq: undersmooth2}
\end{align}
and likewise $\bP(p(x)\in\mathcal{C}_n(x),\forall x\in\mathcal{I}) \geq \bP(\widetilde{W}_n\leq c(\alpha)) - o(1)$, under the conditions specified in Remark \ref{rem: local processes Kolmogorov} where $\widetilde{W}_n$ is defined in Proposition \ref{local:uniform}.
Here the last equality in (\ref{eq: undersmooth2}) follows from the anti-concentration inequality for Gaussian processes (see Lemma \ref{lem: anticoncentration} in the Supplemental Material \cite{CCK14}) together with the fact that $\bE[ \widetilde{W}_{n} ] = O(\sqrt{\log n})$. Hence $\mathcal{C}_n(\cdot)$ is a one-sided uniform confidence band of level $\alpha$ if we set $c(\alpha)$ to be the $(1-\alpha)$-quantile of the distribution of $\widetilde{W}_n$, which in turn can be estimated via  a bootstrap procedure; see our companion paper \cite{CCK13}. Another way is to use a bound on the $(1-\alpha)$-quantile of $\widetilde{W}_{n}$ using sharp deviation inequalities available to Gaussian processes, which leads to analytic construction of confidence bands; see, for example, \cite{CLR12} for this approach. In some applications, the distribution of the approximating Gaussian process is completely known, and in that case the distribution of $\widetilde{W}_{n}$ can be simulated via a direct Monte Carlo method; see \cite{SMD13} for such examples. Finally, we mention that there are alternative, yet more conservative, approaches on construction of confidence bands based on non-asymptotic concentration inequalities (and not on Gaussian approximation); see \cite{LN11} and \cite{KNP12}.
\qed
\end{remark}

\subsection{Series empirical processes}

\label{application-2}

Here we consider the following problem.
Let $(\eta_{1},X_{1}),\dots,(\eta_{n}, X_{n})$ be i.i.d. random variables taking values in the product space $\mathcal{E} \times \R^{d}$, where $(\mathcal{E},\mathcal{A}_{\mathcal{E}})$ is an arbitrary measurable space.
Suppose that the support of $X_{1}$ is normalized to be $[0,1]^{d}$, and for each $K \geq 1$, there are $K$ basis functions $\psi_{K,1},\dots,\psi_{K,K}$ defined on $[0,1]^{d}$.
Let $\psi^{K}(x) = (\psi_{K,1}(x),\dots,\psi_{K,K}(x))^{T}$. Examples of such basis functions are Fourier series, splines, Cohen-Daubechies-Vial (CDV) wavelet bases \cite{CDV93}, Hermite polynomials and so on.
Let $K_{n}$ be a sequence of positive constants such that $K_{n} \to \infty$ as $n \to \infty$. Let  $\mG$ be a class of measurable functions $\mathcal{E} \to \R$ such that $\bE[ g^{2}(\eta_{1}) ] < \infty$ and $\bE[ g(\eta_{1}) \mid X_{1} ] = 0$ a.s. for all $g \in \mG$, and let $\mathcal{I}$ be an arbitrary Borel measurable subset of $[0,1]^{d}$. Suppose that there are sequences of $K_{n} \times K_{n}$ matrices $A_{1n}(g)$ and $A_{2n}(g)$ indexed by $g \in \mG$.
We assume that $s_{\min}(A_{2n}(g)) > 0$ for all $g \in \mG$. In what follows, we let $s_{\min}(A)$ and $s_{\max}(A)$ denote the minimum and maximum singular values of a matrix $A$, respectively.
Consider the following empirical process:
\begin{equation*}
S_{n}(x,g) =  \frac{ \psi^{K_{n}}(x)^{T}A_{1n}(g)^{T}}{|A_{2n}(g)\psi^{K_{n}}(x)|} \left [ \frac{1}{\sqrt{n}}\sum_{i=1}^{n} g( \eta_{i} ) \psi^{K_{n}}(X_{i}) \right ], \ x \in \mathcal{I},  g \in \mG,
\end{equation*}
which we shall call the ``series empirical process'' (we shall formally follow the convention $0/0=0$). The problem here is the Gaussian approximation of the supremum of this series empirical process:
\begin{equation*}
W_{n} := \sup_{(x,g) \in \mathcal{I} \times \mG} S_{n}(x,g).
\end{equation*}

We address this problem in what follows. The study of distributional approximation of this statistic is motivated by inference problems for functions using series (or sieve) estimation. See Examples \ref{example: mean} and \ref{example: quantile} in the Supplemental Material \cite{CCK14} for concrete examples, coming from nonparametric conditional mean and quantile estimation using the series method.
These examples explain and motivate various forms of $S_{n}$ arising in mathematical statistics.

Returning to the general setting, let $B_{n}$ be a centered Gaussian process indexed by $\mathcal{I} \times \mG$ with covariance function
\begin{align}
&\bE [ B_{n}(x,g) B_{n}(\check{x},\check{g}) ] \notag \\
&\quad =  \alpha_{n}(x,g)^{T} \bE[ g(\eta_{1}) \check{g}(\eta_{1})\psi^{K_{n}}(X_{1}) \psi^{K_{n}}(X_{1})^{T} ]\alpha_{n}(\check{x},\check{g}), \label{covfunc2}
\end{align}
where $\alpha_{n}(x,g) = A_{1n}(g)\psi^{K_{n}}(x)/| A_{2n}(g)\psi^{K_{n}}(x) |$. It is expected that under suitable regularity conditions, there is a sequence $\widetilde{W}_{n}$ of random variables such that
$\widetilde{W}_{n} \stackrel{d}{=} \sup_{(x,g) \in \mathcal{I} \times \mG} B_{n}(x,g)$ and as $n \to \infty$, $| W_{n} - \widetilde{W}_{n} | \stackrel{\bP}{\to} 0$.
We shall establish the validity of this approximation with explicit rates.

We make the following assumptions.
\begin{enumerate}
\item[(C1)] $\mG$ is a pointwise measurable VC type class of functions $\mathcal{E} \to \R$ with  measurable envelope $G$  such that $\bE[ g^{2}(\eta_{1}) ] < \infty$ and $\bE[ g(\eta_{1}) \mid X_{1} ] = 0$ a.s. for all $g \in \mG$.
\item[(C2)] There exist some constants $c_{1},C_{1}>0$ such that $s_{\max}(A_{2n}(g)) \leq C_{1}$ and $s_{\min}(A_{2n}(g)) \geq c_{1}$ for all $g \in  \mG$ and $n \geq 1$.
\item[(C3)] $\xi_{n} := \sup_{x \in [0,1]^{d}} | \psi^{K_{n}}(x) | \vee 1 < \infty$ and there exists a constant $C_{2} > 0$ such that $s_{\max}(\bE[ \psi^{K_{n}}(X_{1}) \psi^{K_{n}}(X_{1})^T ]) \leq C_{2}$ for all $n \geq 1$.
The map $(x,g) \mapsto A_{1n}(g)\psi^{K_{n}}(x)/| A_{2n}(g)\psi^{K_{n}}(x) | =: \alpha_{n}(x,g)$ is Lipschitz continuous with Lipschitz constant $\leq L_{n} (\geq 1)$ in the following sense:
\begin{align}
| \alpha_{n}(x,g) - \alpha_{n}(\check{x},\check{g}) | &\leq L_{n} \{ | x - \check{x} | + (\bE[ (g(\eta_{1})-\check{g}(\eta_{1}))^{2} ])^{1/2} \}, \notag \\
&\qquad  \forall x,\check{x} \in [0,1]^{d}, \forall g,\check{g} \in \mG. \label{lipschitz}
\end{align}
Here $\xi_{n}$ and $L_{n}$ are allowed to diverge as $n \to \infty$.
\item[(C4)]  $\log \xi_{n} = O(\log n)$ and $\log L_{n} = O(\log n)$ as $n \to \infty$.
\end{enumerate}

For many commonly used basis functions such as Fourier series, splines and CDV wavelet bases, $\xi_{n} = O(\sqrt{K_{n}})$ as $n \to \infty$; see, for example, \cite{Huang98} and \cite{N97}.
The Lipschitz condition (\ref{lipschitz}) is satisfied  if $\inf_{x \in [0,1]^{d}} | \psi^{K_{n}}(x) | \geq c_{2}>0$, $| \psi^{K_{n}}(x) - \psi^{K_{n}}(\check{x}) | \leq L_{1n} | x - \check{x} |$, and $\| A_{1n}(g) - A_{1n}(\check{g}) \|_{\op}  \vee \| A_{2n}(g) - A_{2n}(\check{g}) \|_{\op} \leq L_{2n} (\bE[ (g(\eta_{1})-\check{g}(\eta_{1}))^{2} ])^{1/2}$, where $c_{2}>0$ is a fixed constant and $L_{1n},L_{2n}$ are sequences of constants possibly divergent as $n \to \infty$ ($\| A \|_{\op}$ denotes the operator norm of a matrix $A$). Then (\ref{lipschitz}) is satisfied with $L_{n}=O(L_{1n} \vee L_{2n})$.
Assumption (C4) states mild growth restrictions on $K_{n}$ and $L_{n}$, and is usually satisfied.

\begin{proposition}[Gaussian approximation to suprema of series empirical processes]
\label{series}
Suppose that assumptions (C1)-(C4) are satisfied. Moreover, suppose either (i) $G$ is bounded (i.e., $\| G \|_{\infty} < \infty$), or (ii)
$\bE[G^{q}(\eta_{1}) ] < \infty$ for some $q \geq 4$ and $\sup_{x \in [0,1]^{d}} \bE[ G^{4}(\eta_{1}) \mid X_{1}=x] < \infty$.
Then for every $n \geq 1$, there is a tight Gaussian random variable $B_{n}$ in $\ell^{\infty}(\mathcal{I} \times \mG)$ with mean zero  and covariance function (\ref{covfunc2}),
and there exists a sequence $\widetilde{W}_{n}$ of random variables such that $\widetilde{W}_{n} \stackrel{d}{=} \sup_{(x,g) \in \mathcal{I} \times \mG} B_{n}(x,g)$ and as $n \to \infty$,
\begin{multline*}
| W_{n} - \widetilde{W}_{n} |   \\
=
\begin{cases}
O_{\bP}\{ n^{-1/6} \xi_{n}^{1/3} \log n + n^{-1/4} \xi_{n}^{1/2} \log^{5/4} n + n^{-1/2} \xi_{n} \log^{3/2} n \},& \text{(i)}, \\
O_{\bP}\{ n^{-1/6} \xi_{n}^{1/3} \log n + n^{-1/4} \xi_{n}^{1/2} \log^{5/4} n + n^{-1/2+1/q} \xi_{n} \log^{3/2} n \}, & \text{(ii)}.
\end{cases}
\end{multline*}
\end{proposition}
\begin{remark}[Discussion and comparisons with other approximations]
Proposition \ref{series} is a new result, and its principal attractive feature is the weak requirement on the number of series functions $K_{n}$ (recall that, for example, for Fourier series, splines, and CDV wavelet bases, we have $\xi_n=O(\sqrt{K_n})$).
Another approach to deduce a result similar to Proposition \ref{series} is to apply Yurinskii's coupling (see Theorem \ref{Yurinskii2} ahead) to random vectors $g(\eta_{i}) \psi^{K_{n}}(X_{i})$, which, however, requires a rather stringent restriction on $K_{n}$, namely $K_{n}^{5}/n \to 0$, for ensuring $| W_{n} - \widetilde{W}_{n} | \stackrel{\bP}{\to} 0$ even in the simplest case where $\mathcal{E}=\R$ and $g(\eta)=\eta$.
See, for example, \cite{CLR12}, Theorem 7. Moreover, the use of Rio's \cite{R94} Theorem 1.1 here is not effective since the total variation bound is large or difficult to control well in this example, which results in restrictive conditions on $K_{n}$ (also Rio's \cite{R94} Theorem 1.1 does not cover case (ii) where $G$ may not be bounded).  \qed
 \end{remark}

\begin{remark}[Converting coupling to convergence in Kolmogorov distance]
As before, we can convert the results in Proposition \ref{series} into convergence of the Kolmogorov distance between the distributions of $W_{n}$ and its Gaussian analogue $\widetilde{W}_{n}$.
Suppose that $\xi_{n}=O(\sqrt{K_{n}})$.
 By Dudley's inequality for Gaussian processes \citep[][Corollary 2.2.8]{VW96}, it is not difficult to deduce that $\bE[ \widetilde{W}_{n} ] = O(\sqrt{\log n})$ under the assumptions of Proposition \ref{series}.  Hence if moreover there exists a constant $\underline{\sigma} > 0$ (independent of $n$) such that $\Var (S_{n}(x,g)) \geq \underline{\sigma}^{2}$ for all $(x,g) \in \mathcal{I} \times \mG$, by Lemma \ref{lem: Kolmogorov2}, we have
\[
| W_{n} - \widetilde{W}_{n} | = o_{\bP}(\log^{-1/2} n) \Rightarrow \sup_{t \in \R} | \bP( W_{n} \leq t ) - \bP (\widetilde{W}_{n} \leq t) | = o(1).
\]
Note that $| W_{n} - \widetilde{W}_{n} | = o_{\bP}(\log^{-1/2} n)$ if $K_{n}(\log n)^{c}/n \to 0$ in case (i) and if $K_{n}(\log n)^{c}/n^{1-2/q} \to 0$ in case (ii), where $c>0$ is some constant.
These requirements on $K_{n}$ are mild, in view of the fact that at least $K_{n}/n \to 0$ is needed for consistency (in the $L^{2}$-norm) of the series estimator \citep[see][]{H03}.
\qed \end{remark}

\begin{remark}[Constructing under-smoothed uniform confidence bands]
Results in Proposition \ref{series} can be used for constructing one- and two-sided uniform confidence bands for various nonparametric functions, such as density, conditional mean, and conditional quantile, estimated via series methods following the same arguments as those described in Remark \ref{eq: local confidence sets} above.
\qed
\end{remark}

\section{A coupling inequality for maxima of sums of random vectors}
\label{coupling}

The main ingredient in the proof of Theorem \ref{main} is a new coupling inequality for maxima of sums of random vectors, which is stated below.

\begin{theorem}[A coupling inequality for maxima of sums of random vectors]
\label{Yurinskii}
Let $X_{1},\dots,X_{n}$ be independent random vectors in $\R^{p}$ with mean zero and finite absolute third moments, that is,  $\bE [X_{ij}]=0$ and $\bE [| X_{ij} |^{3} ] < \infty$ for all $1 \leq i \leq n$ and $1 \leq j \leq p$.
Consider the statistic $Z = \max_{1 \leq j \leq p} \sum_{i=1}^{n} X_{ij}$.
Let $Y_{1},\dots,Y_{n}$ be independent random vectors in $\R^{p}$ with $Y_{i} \sim N(0,\bE[ X_{i}X_{i}^{T}])$, $1 \leq i \leq n$.
Then for every $\beta > 0$ and $\delta > 1/\beta$, there exists a random variable $\widetilde{Z}\stackrel{d}{=} \max_{1 \leq j \leq p} \sum_{i=1}^{n} Y_{ij}$ such that
\begin{equation*}
\bP ( | Z - \widetilde{Z} | >  2 \beta^{-1} \log p +  3\delta )  \leq \frac{\varepsilon + C\beta \delta^{-1}\{ B_{1} +\beta (B_{2}+B_{3})\}}{1-\varepsilon},
\end{equation*}
where $\varepsilon = \varepsilon_{\beta,\delta}$ is given by
\begin{equation}
\varepsilon = \sqrt{e^{-\alpha}(1+\alpha)} < 1, \ \alpha = \beta^{2} \delta^{2}  -1 > 0, \label{eq: epsilon}
\end{equation}
and
\begin{align*}
B_{1} &= \bE \left[ \max_{1 \leq j,k \leq p} | \sum_{i=1}^{n} (X_{ij} X_{ik} - \bE[ X_{ij} X_{ik} ]) | \right], \\
B_{2} &=  \bE \left[  \max_{1 \leq j \leq p} \sum_{i=1}^{n} | X_{ij} |^{3} \right], \\
B_{3} &=\sum_{i=1}^{n} \bE \left [  \max_{1 \leq j \leq p}  | X_{ij} |^{3} \cdot 1 \left (\max_{1 \leq j \leq p} | X_{ij} | > \beta^{-1}/2 \right ) \right].
\end{align*}
\end{theorem}

A different, though related, Gaussian approximation inequality was obtained in Theorem 2.1 of \cite{CCK12} with different techniques.
We have chosen to present a new theorem here because 1) it is based on the Stein's exchangeable pairs technique, which is well understood in the literature, and our theorem might be helpful for deriving further results in the future; 2)
applying Theorem 2.1 of \cite{CCK12} here would require solving a complicated optimization problem to find the best bound for the coupling problem; and 3)
our new theorem does not require truncating normal random vectors, allowing us to avoid an additional layer of complication in the final application to empirical processes.

 The following corollary is useful for many applications. Recall $n \geq 3$.

\begin{corollary}[An applied coupling inequality for maxima of sums of random vectors]
\label{cor:Yurinskii}
Consider the same setup as in Theorem \ref{Yurinskii}.
Then for every $\delta > 0$, there exists a random variable $\widetilde{Z}\stackrel{d}{=} \max_{1 \leq j \leq p} \sum_{i=1}^{n} Y_{ij}$ such that
\begin{equation}
\bP ( | Z - \widetilde{Z} | >  16\delta )  \lesssim \delta^{-2} \{ B_{1} + \delta^{-1} ( B_{2} + B_{4})\log (p \vee n) \} \log (p \vee n) + \frac{\log n}{n}, \label{stein-coupling}
\end{equation}
where $B_{1}$ and $B_{2}$ are as in Theorem \ref{Yurinskii}, and
\begin{equation*}
B_{4} = \sum_{i=1}^{n} \bE \left[  \max_{1 \leq j \leq p}  | X_{ij} |^{3} \cdot  1 \left (\max_{1 \leq j \leq p} | X_{ij} | > \delta/\log (p \vee n) \right ) \right ].
\end{equation*}
\end{corollary}

\begin{proof}[Proof of Corollary \ref{cor:Yurinskii}]
In Theorem \ref{Yurinskii}, take $\beta = 2 \delta^{-1}\log (p \vee n)$. Then $\alpha = \beta^{2} \delta^{2} - 1 = 4 \log^{2}(p \vee n) - 1 \geq 2 \log (p \vee n)$ (recall  $n \geq 3 > e$), so that $\varepsilon \leq 2 \log (p \vee n)/(p \vee n) \leq 2 n^{-1} \log n$. This completes the proof.
\end{proof}


Theorem \ref{Yurinskii} is a coupling inequality similar in nature to Yurinskii's \cite{Y77} coupling for sums of random vectors (as opposed to the maxima of such vectors as in the current theorem).  Before proving Theorem \ref{Yurinskii}, let us first recall Yurinskii's coupling inequality.

\begin{theorem}[Yurinskii's coupling for sums of random vectors; \cite{Y77}; see also \cite{LC88}]
\label{Yurinskii2}
Consider the same setup as in Theorem \ref{Yurinskii}. Let $S_{n} = \sum_{i=1}^{n} X_{i}$. Then for every $\delta > 0$, there exists a random vector $T_{n} \stackrel{d}{=} \sum_{i=1}^{n} Y_{i}$ such that
\begin{equation*}
\bP (| S_{n} - T_{n} | > 3\delta) \lesssim B_{0} \left ( 1 + \frac{| \log (1/B_{0}) |}{p} \right ),
\end{equation*}
where $B_{0} = p \delta^{-3} \sum_{i=1}^{n} \bE[ | X_{i} |^{3} ]$.
\end{theorem}

For the proof, see \cite{P02}, Section 10.4.  Because of the general fact that $\max_{1 \leq j \leq n} | x_{j} | \leq | x |$ for $x \in \R^{p}$, one has
\begin{equation*}
| \max_{1 \leq j \leq p} (S_{n})_{j} - \max_{1 \leq j \leq n} (T_{n})_{j} | \leq \max_{1 \leq j \leq p} | (S_{n} - T_{n})_{j} | \leq | S_{n} - T_{n} |.
\end{equation*}
Hence if we take $\widetilde{Z} = \max_{1 \leq j \leq p} (T_{n})_{j}$,
\begin{equation}
\label{crude}
\bP (| Z - \widetilde{Z} | > 3\delta) \lesssim B_{0} \left ( 1 + \frac{| \log (1/B_{0}) |}{p} \right ).
\end{equation}
Unfortunately, when $p$ is large, the right side needs not be small. This is because $B_{0}$ is proportional to $\sum_{i=1}^{n} \bE[ | X_{i} |^{3} ]$ and this quantity may be larger than what we want.

To better understand the difference between (\ref{stein-coupling}) and (\ref{crude}), consider the situation where $p$ is indexed by $n$ and $p=p_{n} \to \infty$ as $n \to \infty$.
Moreover, consider the simple case where $X_{ij} = x_{ij}/\sqrt{n}$ and $| x_{ij} | \leq b$ ($x_{ij}$ are random; $b$ is a fixed constant).
Then $B_{1} = O( n^{-1/2} \log^{1/2} p_{n}), \ B_{2} + B_{4}= O(n^{-1/2})$.
The former estimate is deduced from the fact that, using the symmetrization and the maximal inequality for Rademacher averages conditional on $X_{1},\dots,X_{n}$ \citep[use][Lemmas 2.2.2 and 2.2.7]{VW96}, one has
$B_{1} \lesssim \sqrt{\log (1+p)} \bE[ \max_{1 \leq j \leq p} ( \sum_{i=1}^{n} X_{ij}^{4} )^{1/2}  ]$.
On the other hand,
$p_{n} \sum_{i=1}^{n} | X_{i} |^{3} = O( n^{-1/2} p_{n}^{5/2} ).$
Therefore, to make $| Z - \widetilde{Z} | \stackrel{\bP}{\to} 0$,  the former (\ref{stein-coupling}) allows $p_{n}$ to be of an exponential order ($p_{n}$ can be as large as $\log p_{n} = o(n^{1/4})$; hence, for example, $p_{n}$ can be of order $e^{n^{\alpha}}$ for $0 < \alpha < 1/4$), while the latter (\ref{crude}) restricts $p_{n}$ to be $p_{n} = o(n^{1/5})$. Note that, under the exponential moment condition,  instead of Yurinskii's coupling, we can use Zaitsev's coupling inequality \citep[][Theorem 1.1]{Z87} but it still requires $p_{n} = o(n^{1/5})$ to deduce that $| Z - \widetilde{Z} | \stackrel{\bP}{\to} 0$ (although by using Zaitsev's coupling, we indeed have an exponential type inequality for $| Z - \widetilde{Z} |$).

\begin{remark}[Connection to Theorem \ref{main}]
The importance of Theorem \ref{Yurinskii} in the context of the proof of Theorem \ref{main} is described as follows.
In the proof of Theorem \ref{main}, we  make a finite approximation of $\mF$ by a minimal $\varepsilon \| F \|_{P,2}$-net of $(\mF,e_{P})$ and apply Theorem \ref{Yurinskii} to the ``discretized'' empirical process; hence in this application, $p =N(\mF,e_{P},\varepsilon \| F \|_{P,2})$. The fact that Theorem \ref{Yurinskii} allows for ``large'' $p$ means that a ``finer'' discretization is possible, and as a result, the bound in Theorem \ref{main} depends on the covering number $N(\mF,e_{P},\varepsilon \| F \|_{P,2})$ only through its logarithm:  $\log N(\mF,e_{P},\varepsilon \| F \|_{P,2})$. \qed
\end{remark}

We will use a version of Strassen's theorem to prove Theorem \ref{Yurinskii}. We state it for the reader's convenience. The proof of this result can be found in the Supplemental Material \cite{CCK14}.

\begin{lemma}[An implication of Strassen's theorem]
\label{strassen}
Let $\mu$ and $\nu$ be Borel probability measures on $\R$, and let $V$ be a random variable defined on a probability space $(\Omega,\mathcal{A},\bP)$ with distribution $\mu$.
Suppose that the probability space $(\Omega,\mathcal{A},\bP)$ admits a uniform random variable on $(0,1)$ independent of $V$. Let $\varepsilon > 0$ and $\delta > 0$  be two positive constants.
Then
there exists a random variable $W$,  defined on $(\Omega,\mathcal{A},\bP)$, with distribution $\nu$ such that $\bP ( | V-W | > \delta ) \leq \varepsilon$ if and only if $\mu (A) \leq \nu(A^{\delta}) + \varepsilon$ for every Borel subset $A$ of $\R$.
\end{lemma}

\begin{proof}[Proof of Theorem \ref{Yurinskii}]
For the notational convenience, write $e_{\beta} = \beta^{-1} \log p$.
Construct  $Y_{1},\dots,Y_{n}$ independent of $X_{1},\dots,X_{n}$.
By Lemma \ref{strassen}, the conclusion follows if we can prove that for every Borel subset $A$ of $\R$,
\begin{equation*}
\bP (Z \in A) \leq \bP(\widetilde{Z}^{*} \in A^{2e_{\beta} + 3\delta}) + \frac{\varepsilon + C\beta \delta^{-1}\{ B_{1} + \beta( B_{2} +B_{3}) \}}{1-\varepsilon},
\end{equation*}
where $\widetilde{Z}^{*} := \max_{1 \leq j \leq p} \sum_{i=1}^{n} Y_{ij}$.
Let $S_{n} = \sum_{i=1}^{n} X_{i}$ and $T_{n} = \sum_{i=1}^{n} Y_{i}$.
Fix any Borel subset $A$ of $\R$.
We divide the proof into several steps.

{\bf Step 1}:
We approximate the non-smooth map $x \mapsto 1_{A}(\max_{1 \leq j \leq p}x_{j})$ by a smooth function.
The first step is to approximate the map $x \mapsto \max_{1 \leq j \leq p}x_{j}$ by a smooth function.
Consider  the function $F_{\beta} : \R^{p} \to \R$ defined by $F_{\beta}(x) = \beta^{-1} \log (\sum_{j=1}^{p} e^{\beta x_{j}} )$,
which gives a smooth  approximation of  $\max_{1 \leq j \leq p} x_{j}$; this function arises in definition of
free energy in spin glasses \cite{panchenko}. Indeed, an elementary calculation gives the following inequality: for every $x = (x_{1},\dots,x_{p})^{T} \in \R^{p}$,
\begin{equation}
\max_{1 \leq j \leq p} x_{j} \leq F_{\beta} (x) \leq \max_{1 \leq j \leq p} x_{j} + \beta^{-1} \log p. \label{chat}
\end{equation}
See \cite{C05b}. Hence  we have
\begin{equation*}
\bP ( Z \in A ) \leq \bP ( F_{\beta}(S_{n}) \in A^{e_{\beta}}) =\bE[ 1_{A^{e_{\beta}}} (F_{\beta}(S_{n})) ].
\end{equation*}

{\bf Step 2}:
The next step is to approximate the indicator function $t \mapsto 1_{A}(t)$ by a smooth function. This step is rather standard.

\begin{lemma}
\label{approx}
Let $\beta > 0$ and $\delta > 1/\beta$. For every Borel subset $A$ of $\R$,
there exists a smooth function $g: \R \to \R$ such that $\| g' \|_{\infty} \leq \delta^{-1}, \| g'' \|_{\infty} \leq C \beta \delta^{-1}, \| g''' \|_{\infty} \leq C \beta^{2} \delta^{-1}$, and
\begin{equation*}
(1-\varepsilon) 1_{A}(t) \leq g(t) \leq \varepsilon + (1-\varepsilon) 1_{A^{3\delta}} (t), \ \forall t \in \R,
\end{equation*}
where $\varepsilon = \varepsilon_{\beta,\delta}$ is given by (\ref{eq: epsilon}).
\end{lemma}
\begin{proof}[Proof of Lemma \ref{approx}]
The proof is due to \cite{P02}, Lemma 10.18 (p. 248). Let $\rho(\cdot,\cdot)$ denote the Euclidean distance on $\R$.
Then consider the function $h(t) = (1-\rho(t,A^{\delta})/\delta)_{+}$.
Note that $h$ is Lipschitz continuous with Lipschitz constant $\leq \delta^{-1}$.
Construct a smooth approximation of $h(t)$ by
\begin{equation*}
g(t) = \frac{\beta}{\sqrt{2\pi}} \int_{\R} h(s) e^{-\frac{1}{2} \beta^{2}(s-t)^{2}}d s =  \frac{1}{\sqrt{2\pi}} \int_{\R} h(t+ \beta^{-1} z) e^{-\frac{1}{2} z^{2}}  d z.
\end{equation*}
Then the map $t \mapsto g(t)$ is infinitely differentiable, and
\begin{equation*}
\| g' \|_{\infty} \leq \delta^{-1}, \ \| g'' \|_{\infty} \leq C \beta \delta^{-1},  \ \| g''' \|_{\infty} \leq C \beta^{2} \delta^{-1}.
\end{equation*}
The rest of the proof is the same as \cite{P02}, Lemma 10.18 and omitted.
\end{proof}

Apply Lemma \ref{approx} to $A=A^{e_{\beta}}$ to construct  a  suitable function $g$. Then
\begin{equation*}
\bE[  1_{A^{e_{\beta}}} (F_{\beta}(S_{n})) ] \leq  (1-\varepsilon)^{-1} \bE [g \circ F_{\beta}(S_{n}) ].
\end{equation*}

{\bf Step 3}:
The next step uses Stein's method to compare $\bE [g \circ F_{\beta}(S_{n}) ]$ and $\bE [g \circ F_{\beta}(T_{n}) ]$.
The following argument is inspired by \cite{CM08}, Theorem 7.
We first make some complimentary computations.  Here for a smooth function $f: \R^{p} \to \R$, we use the notation $\partial_{j} f(x) = \partial f(x)/\partial x_{j}$, $\partial_{j} \partial_{k} f(x) = \partial^{2} f(x)/\partial x_{j}\partial x_{k}$, and so on.

\begin{lemma}
\label{computation}
Let $\beta > 0$. For every $g \in C^{3}(\R)$,
\begin{align}
&\sum_{j,k=1}^{p} | \partial_{j} \partial_{k} (g \circ F_{\beta})(x) | \leq  \| g'' \|_{\infty} + 2 \| g' \|_{\infty} \beta, \label{ineq1} \\
&\sum_{j,k,l=1}^{p} | \partial_{j} \partial_{k}\partial_{l} (g \circ F_{\beta})(x) | \leq \| g''' \|_{\infty} + 6 \| g'' \|_{\infty} \beta + 6 \| g' \|_{\infty} \beta^{2}. \label{ineq2}
\end{align}
Moreover, let $U_{jkl}(x):= \sup \{ | \partial_{j} \partial_{k}\partial_{l} (g \circ F_{\beta})(x + y) | : y \in \R^{p}, | y_{j} | \leq \beta^{-1}, 1 \leq \forall j \leq p \}$. Then
\begin{equation}
\sum_{j,k,l=1}^{p} U_{jkl}(x) \leq C (\| g''' \|_{\infty} + \| g'' \|_{\infty} \beta +  \| g' \|_{\infty} \beta^{2}).  \label{switching}
\end{equation}
\end{lemma}

\begin{proof}[Proof of Lemma \ref{computation}]
Let $\delta_{jk} = 1(j = k)$. A direct calculation gives
\begin{equation*}
\partial_j F_{\beta}(x) = \pi_{j}(z), \
\partial_j \partial_k F_{\beta}(x) = \beta w_{jk} (x), \
\partial_j \partial_k \partial_l F_{\beta}(x) = \beta^2 q_{jkl}(x),
\end{equation*}
where
\begin{align*}
&\pi_{j}(x) =    e^{\beta x_j}/{\textstyle \sum}_{k=1}^{p}e^{\beta x_{k}}, \ w_{jk}(x) = ( \pi_{j} \delta_{jk} - \pi_{j} \pi_{k}) (x),  \\
&q_{jkl}(x) = (\pi_{j}\delta_{jl}\delta_{jk}-\pi_{j}\pi_{l}\delta_{jk} -\pi_{j}\pi_{k}(\delta_{jl}+\delta_{kl})    + 2\pi_{j}\pi_{k}\pi_{l})(x).
\end{align*}
By these expressions, we have
\begin{equation*}
\pi_{j}(x) \geq 0,  \ \sum_{j=1}^p \pi_{j} (x) = 1, \ \sum_{j,k=1}^p |w_{jk}(x)|  \leq 2, \ \sum_{j,k,l=1}^p |q_{jkl}(x)| \leq 6 .
\end{equation*}
Inequalities (\ref{ineq1}) and (\ref{ineq2}) follow from these relations and the following computation.
\begin{align*}
& \partial_{j}  (g \circ F_{\beta})(x) =  (g' \circ F_{\beta})(x) \pi_{j} (x), \\
& \partial_{j} \partial_{k}  (g \circ F_{\beta})(x) = (g'' \circ F_{\beta})(x)  \pi_{j}(x) \pi_{k}(x) +  (g' \circ F_{\beta})(x) \beta w_{jk}(x),\\
& \partial_{j} \partial_{k}  \partial_{l}  (g \circ F_{\beta})(x) = (g''' \circ F_{\beta})(x)  \pi_{j}(x) \pi_{k}(x)\pi_{l}(x) \\
&\qquad \qquad \qquad  + (g'' \circ F_{\beta})(x)  \beta (w_{jk} (x) \pi_{l}(x) + w_{jl} (x) \pi_{k}(x) + w_{kl}(x) \pi_{j}(x)) \\
&\qquad \qquad \qquad    +(g' \circ F_{\beta})(x)  \beta^2 q_{jkl}(x). \ \
\end{align*}

For the last inequality (\ref{switching}), it is standard to see that whenever $| y_{j} | \leq \beta^{-1}, 1 \leq \forall j \leq p$, we have $\pi_{j}(x + y) \leq e^{2} \pi_{j}(x)$,
from which the desired inequality follows.
\end{proof}

For $i=1,\dots,n$, let $X_{i}'$ be an independent copy of $X_{i}$. Let $I$ be a uniform random variable on $\{1,\dots, n \}$ independent of all the other variables.
Define $S_{n}':= S_{n} - X_{I} + X'_{I}.$ For $\lambda \in \R^{p}$,
\begin{align*}
&\bE[ e^{\sqrt{-1} \lambda^{T} S_{n}^{'}} ]
=  \frac{1}{n} \sum_{i=1}^{n}\bE[ e^{\sqrt{-1} \lambda^{T} (S_{n} - X_{i})} ] \bE [ e^{\sqrt{-1} \lambda^{T}X'_{i}} ] \\
&= \frac{1}{n} \sum_{i=1}^{n}  \prod_{j \neq i} \bE[ e^{\sqrt{-1} \lambda^{T} X_{j}} ]  \bE [ e^{\sqrt{-1} \lambda^{T}X_{i}} ]
= \prod_{i = 1}^{n} \bE[ e^{\sqrt{-1} \lambda^{T} X_{i}} ] = \bE[ e^{\sqrt{-1} \lambda^{T}S_{n}} ].
\end{align*}
Hence $S'_{n} \stackrel{d}{=} S_{n}$.
Also with $X_{1}^{n} = \{ X_{1},\dots,X_{n} \}$,
\begin{equation}
\bE[ S_{n}' - S_{n} \mid  X_{1}^{n}] =  \bE[ X'_{I} - X_{I} \mid X_{1}^{n} ] = -n^{-1} S_{n},  \label{stein1}
\end{equation}
and
\begin{align}
&\bE [ (S'_{n} - S_{n})  (S'_{n} - S_{n})^{T} \mid X_{1}^{n}]
= \bE [ (X'_{I} - X_{I})  (X'_{I} - X_{I})^{T} \mid X_{1}^{n} ] \notag \\
&=\frac{1}{n} \sum_{i=1}^{n} \bE [ (X'_{i} - X_{i} )  (X'_{i} - X_{i})^{T} \mid X_{1}^{n}]
=\frac{1}{n} \sum_{i=1}^{n} (\bE [ X_{i} X_{i}^{T} ] + X_{i} X_{i}^{T}) \notag \\
&=  \frac{2}{n} \sum_{i=1}^{n}  \bE [ X_{i} X_{i}^{T} ] +  \frac{1}{n} \sum_{i=1}^{n} ( X_{i} X_{i}^{T} - \bE [ X_{i} X_{i}^{T} ]) \notag \\
&=  \frac{2}{n} \sum_{i=1}^{n}  \bE [ X_{i} X_{i}^{T} ] + n^{-1} V, \label{stein2}
\end{align}
where $V$ is the $p \times p$ matrix defined by $V = (V_{jk})_{1 \leq j,k \leq p} = \sum_{i=1}^{n} ( X_{i} X_{i}^{T} - \bE [ X_{i} X_{i}^{T} ])$.

For the notational convenience, write $f = g \circ F_{\beta}$. Consider
\begin{equation*}
h(x) = \int_{0}^{1} \frac{1}{2t} \bE[ f (\sqrt{t} x + \sqrt{1-t} T_{n}) - f(T_{n})] dt.
\end{equation*}
Then Lemma 1 of \cite{M09} implies
\begin{equation*}
\sum_{j=1}^{p} x_{j} \partial_{j} h(x) -  \sum_{j,k=1}^{p}  \sum_{i=1}^{n}\bE [ X_{ij} X_{ik} ]  \partial_{j} \partial_{k} h(x) = f(x) - \bE [ f(T_{n}) ],
\end{equation*}
and especially
\begin{align}
\bE[ f(S_{n}) ] - \bE [ f(T_{n}) ] &= \bE \left[ \sum_{j=1}^{p}  \sum_{i=1}^{n}X_{ij} \partial_{j} h(S_{n}) \right ] \notag \\
&\quad - \bE \left [ \sum_{j,k=1}^{p}   \sum_{i=1}^{n} \bE [ X_{ij} X_{ik} ] \partial_{j} \partial_{k} h(S_{n}) \right ]. \label{stein3}
\end{align}

Denote by $\nabla h(x)$ and $\Hess h(x)$ the gradient vector and the Hessian matrix of $h(x)$, respectively.
Let
\begin{align*}
R&= h(S'_{n}) - h(S_{n}) -  (S_{n}'- S_{n})^{T} \nabla h(S_{n}) \\
&\quad - 2^{-1}(S_{n}'- S_{n})^{T} (\Hess h(S_{n})) (S_{n}'- S_{n}).
\end{align*}
Then  one has
\begin{align*}
0 &= n \bE[ h(S'_{n}) -  h(S_{n}) ] \quad (\text{as} \ S_{n}' \stackrel{d}{=} S_{n}) \\
&= n \bE [ (S_{n}'- S_{n})^{T} \nabla h(S_{n}) + 2^{-1}(S_{n}'- S_{n})^{T} (\Hess h(S_{n})) (S_{n}'- S_{n}) + R ] \\
&= n \bE  \Big [ \bE [ (S_{n}'- S_{n})^{T} \mid X_{1}^{n} ] \nabla h(S_{n})  \\
&\qquad \qquad + 2^{-1} \Tr \Big( (\Hess h(S_{n})) \bE [ (S_{n}'- S_{n})(S_{n}'- S_{n})^{T} \mid X_{1}^{n} ] \Big)  + R \Big ]  \\
&= \bE \left [ - \sum_{j=1}^{p} \sum_{i=1}^{n}  X_{ij} \partial_{j}h (S_{n}) +  \sum_{j,k=1}^{p}  \sum_{i=1}^{n} \bE [ X_{ij} X_{ik} ] \partial_{j} \partial_{k} h(S_{n}) \right  ] \\
&\qquad  + \bE \left [\frac{1}{2}   \sum_{j,k=1}^{p} V_{jk} \partial_{j} \partial_{k} h(S_{n}) + n R \right ] \quad (\text{by (\ref{stein1}) and (\ref{stein2})}) \\
&=-  \bE[ f(S_{n}) ] + \bE [ f(T_{n}) ] + \bE \left [\frac{1}{2} \sum_{j,k=1}^{p} V_{jk} \partial_{j} \partial_{k} h(S_{n}) + n R \right ], \quad (\text{by (\ref{stein3})})
\end{align*}
that is,
\begin{equation*}
\bE[ f(S_{n}) ] - \bE [ f(T_{n}) ] = \bE \left [\frac{1}{2} \sum_{j,k=1}^{p} V_{jk} \partial_{j} \partial_{k} h(S_{n}) + n R \right ].
\end{equation*}
Using Lemma \ref{computation}, one has
\begin{align*}
| \sum_{j,k=1}^{p} V_{jk} \partial_{j} \partial_{k} h(S_{n}) | \leq \max_{1 \leq j,k \leq p} | V_{jk} | \sum_{j,k=1}^{p} | \partial_{j} \partial_{k} h(S_{n}) | \leq C \beta \delta^{-1}  \max_{1 \leq j,k \leq p} | V_{jk} |,
\end{align*}
and with $\Delta_{i} := (\Delta_{i1},\dots,\Delta_{ip})^{T} := X_{i}'-X_{i}$,
\begin{align}
| \bE[  n R ] |  &= \left | \bE \left  [  \frac{1}{2} \sum_{i=1}^{n}\sum_{j,k,l=1}^{p}   \Delta_{ij}\Delta_{ik} \Delta_{il} (1-\theta)^{2} \partial_{j}  \partial_{k} \partial_{l} h(S_{n} + \theta \Delta_{i})  \right ] \right | \notag \\
&\qquad \qquad  (\theta \sim U(0,1) \ \text{independent of all the other variables}) \notag \\
 &\leq \frac{1}{2} \bE \left [   \sum_{i=1}^{n}\sum_{j,k,l=1}^{p}  | \Delta_{ij} \Delta_{ik} \Delta_{il} | \cdot |  \partial_{j}  \partial_{k} \partial_{l} h(S_{n} + \theta \Delta_{i})  | \right]. \label{third}
\end{align}
Let $\chi_{i} = 1( \max_{1 \leq j \leq p} | \Delta_{ij} | \leq \beta^{-1} )$ and $\chi_{i}^{c} :=1-\chi_{i}$. Then
\begin{align*}
\text{(\ref{third})} = \frac{1}{2} \bE \left [ \sum_{i=1}^{n} \chi_{i} * \right ] + \frac{1}{2} \bE \left [ \sum_{i=1}^{n} \chi_{i}^{c} * \right ] =: \frac{1}{2} \left [ ({\rm A}) + ({\rm B}) \right ].
\end{align*}
Observe that
\begin{align*}
&({\rm A}) \leq  \bE \left [    \sum_{j,k,l=1}^{p} \max_{1 \leq i \leq n} ( \chi_{i} \cdot |  \partial_{j}  \partial_{k} \partial_{l} h(S_{n} + \theta \Delta_{i})  | ) \times \max_{1 \leq j,k,l \leq p}  \sum_{i=1}^{n} |  \Delta_{ij} \Delta_{ik} \Delta_{il} |  \right ] \notag \\
&\leq C \beta^{2} \delta^{-1}\bE \left [ \max_{1 \leq j,k,l \leq p}  \sum_{i=1}^{n} |  \Delta_{ij} \Delta_{ik} \Delta_{il} |  \right ]  \quad (\text{by (\ref{switching})}) \\
&\leq C \beta^{2} \delta^{-1}\bE \left [\max_{1 \leq j  \leq p}  \sum_{i=1}^{n} | \Delta_{ij} |^{3} \right ]  \leq C\beta^{2} \delta^{-1}\bE \left [ \max_{1 \leq j \leq p} \sum_{i=1}^{n}| X_{ij} |^{3} \right ] = C\beta^{2} \delta^{-1}B_{2},
\end{align*}
and
\begin{align*}
({\rm B}) &\leq C\beta^{2} \delta^{-1}\sum_{i=1}^{n} \bE \left  [ \chi_{i}^{c} \max_{1 \leq j \leq p}| \Delta_{ij} |^{3} \right ]   \quad (\text{by (\ref{ineq2})}) \\
&\leq C\beta^{2} \delta^{-1}\sum_{i=1}^{n} \bE \left [\chi_{i}^{c}\max_{1 \leq j \leq p} | X_{ij} |^{3} \right ]. \quad (\text{by symmetry})
\end{align*}
As $\chi_{i}^{c} \leq 1 ( \max_{1 \leq j \leq p} | X_{ij} | > \beta^{-1}/2 ) + 1  ( \max_{1 \leq j \leq p} | X'_{ij} | > \beta^{-1}/2  )$,
we have
\begin{align}
\bE \left [ \chi_{i}^{c}\max_{1 \leq j \leq p} | X_{ij} |^{3} \right ] &\leq \bE \left  [ \max_{1 \leq j \leq p} | X_{ij} |^{3} \cdot  1 \left ( \max_{1 \leq j \leq p} | X_{ij} | > \beta^{-1}/2 \right ) \right ] \notag \\
&\quad + \bE \left [ \max_{1 \leq j \leq p} | X_{ij} |^{3} \right ] \cdot \bP \left ( \max_{1 \leq j \leq p} | X_{ij} | > \beta^{-1}/2 \right ).  \label{ineq-x}
\end{align}
We here recall  Chebyshev's association inequalities stated in the following lemma. For a proof, see, for example, Theorem 2.14 in \cite{BLM13}.
\begin{lemma}[Chebyshev's association inequalities]
\label{monotone}
Let $\varphi$ and $\psi$ be functions defined on an interval $\mathcal{I}$ in $\R$, and let $\xi$ be a random variable such that $\bP ( \xi \in \mathcal{I} ) = 1$. Suppose that $\bE [ | \varphi (\xi) |] < \infty, \bE [ | \psi (\xi) |] < \infty$ and $\bE [ | \varphi(\xi) \psi(\xi) |] < \infty$.
Then $\Cov (\varphi(\xi), \psi (\xi) ) \geq 0$ if $\varphi$ and $\psi$ are monotone in the same direction, and $\Cov (\varphi(\xi), \psi (\xi) ) \leq 0$  if $\varphi$ and $\psi$ are monotone in the opposite direction.
\end{lemma}
Since the maps $t \mapsto t^{3}$ and $t \mapsto 1(t > \beta^{-1}/2)$ are non-decreasing on $[0,\infty)$, the second term on the right side of (\ref{ineq-x}) is not larger than the first term. Hence
\begin{equation*}
({\rm B}) \leq C \beta^{2} \delta^{-1} \sum_{i=1}^{n} \bE \left [ \max_{1 \leq j \leq p} | X_{ij} |^{3}\cdot  1 \left ( \max_{1 \leq j \leq p} | X_{ij} | > \beta^{-1}/2 \right ) \right ] = C\beta^{2} \delta^{-1}B_{3}.
\end{equation*}
Therefore, we conclude that \[
| \bE[ f(S_{n}) ] - \bE [ f(T_{n}) ]  | \leq C  \beta \delta^{-1} \{ B_{1}  + \beta( B_{2}  + B_{3})\}.
\]

{\bf Step 4}: Combining Steps 1-3, one has
\begin{align*}
&\bP ( Z \in A ) \leq (1-\varepsilon)^{-1} \bE [ g\circ F_{\beta}(T_{n}) ] + \frac{C\beta \delta^{-1}\{ B_{1}  +\beta( B_{2}  + B_{3})\}}{1-\varepsilon} \\
&\quad \leq \bP (  F_{\beta}(T_{n}) \in A^{e_{\beta} + 3\delta} ) + \frac{\varepsilon + C\beta \delta^{-1}\{ B_{1} +\beta( B_{2}  + B_{3})\}}{1-\varepsilon} \\
&\qquad \qquad \qquad \qquad \qquad \qquad \qquad \qquad \qquad  \qquad  (\text{by construction of $g$}) \\
&\quad \leq \bP (  \widetilde{Z}^{*}  \in A^{2e_{\beta} + 3\delta} ) + \frac{\varepsilon + C\beta \delta^{-1}\{ B_{1} +\beta( B_{2}  + B_{3})\}}{1-\varepsilon}. \quad  (\text{by (\ref{chat})})
\end{align*}
This completes the proof.
\end{proof}

\section{Inequalities for empirical processes}
\label{empirical}

In this section, we shall present some inequalities for empirical processes that will be used in the proofs of Theorem \ref{main} and Lemma \ref{lem1}. These inequalities are of interest in their own rights.
Consider the same setup as in Section \ref{sec:main}, that is,  let $X_{1},\dots,X_{n}$ be i.i.d. random variables taking values in a measurable space $(S,\mS)$ with common distribution $P$. Let $\mF$ be a pointwise measurable class of functions $S \to \R$, to which a measurable envelope $F$ is attached. In this section, however, {\em we do not assume that $\mF$ is $P$-centered}.
Consider the empirical process $\bG_{n} f = n^{-1/2} \sum_{i=1}^{n} (f(X_{i}) - Pf)$.
Let $\sigma^{2} > 0$ be any positive constant such that $\sup_{f \in \mF} Pf^{2} \leq \sigma^{2} \leq \| F \|_{P,2}^{2}$.
Let $M = \max_{1 \leq i \leq n} F(X_{i})$.

\begin{theorem}[A useful deviation inequality for suprema of empirical processes]
\label{concentration}
Suppose that $F \in \mL^{q} ( P )$ for some $q \geq 2$.
Then for every $t \geq 1$, with probability $> 1-t^{-q/2}$,
\begin{multline*}
\| \bG_{n} \|_{\mF} \leq (1+\alpha) \bE [ \| \bG_{n} \|_{\mF} ] + K(q) \Big [ (\sigma + n^{-1/2} \| M \|_{q}) \sqrt{t} \\
+  \alpha^{-1}  n^{-1/2} \| M \|_{2}t \Big ], \ \forall \alpha > 0,
\end{multline*}
where $K(q) > 0$ is a constant depending only on $q$.
\end{theorem}


\begin{proof}[Proof of Theorem \ref{concentration}]
The theorem essentially follows from \cite{BBLM05}, Theorem 12, which states that
\begin{equation*}
\| (\| \bG_{n} \|_{\mF} - \bE [\| \bG_{n} \|_{\mF}])_{+} \|_{q} \lesssim \sqrt{q} (\Sigma + \sigma) +q n^{-1/2} (\| M \|_{q} + \sigma),
\end{equation*}
where $\Sigma^{2} = \bE [ \| n^{-1} \sum_{i=1}^{n} (f(X_{i})-Pf)^{2} \|_{\mF} ]$. By Lemma 7 of the same paper,
\begin{equation*}
\Sigma^{2} \leq \sigma^{2} + 64 n^{-1/2} \| M \|_{2} \bE [ \| \bG_{n} \|_{\mF} ] + 32 n^{-1} \| M \|_{2}^{2}.
\end{equation*}
Hence, using the simple inequality $2 \sqrt{ab} \leq \beta a+ \beta^{-1}b,  \forall \beta > 0$, one has
\begin{align*}
\| (\| \bG_{n} \|_{\mF} - \bE [\| \bG_{n} \|_{\mF}])_{+} \|_{q} &\lesssim \sqrt{q} \beta \bE [ \| \bG_{n} \|_{\mF} ]  + \sqrt{q} (1+\beta^{-1}) n^{-1/2} \| M \|_{2} \\
&\quad + \sqrt{q} \sigma + q n^{-1/2} (\| M \|_{q} + \sigma).
\end{align*}
Therefore, by Markov's inequality, for every $t \geq 1$, with probability $>1-t^{-q}$,
\begin{align*}
&\| \bG_{n} \|_{\mF} \leq  \bE [\| \bG_{n} \|_{\mF}] +  (\| \bG_{n} \|_{\mF} - \bE [\| \bG_{n} \|_{\mF}])_{+} \\
&\leq (1 + C \sqrt{q} \beta t ) \bE [\| \bG_{n} \|_{\mF}]  + C\sqrt{q} (1+\beta^{-1}) n^{-1/2} \| M \|_{2} t \\
&\quad + C \sqrt{q} \sigma t+ C q n^{-1/2} (\| M \|_{q} + \sigma) t, \ \forall \beta > 0.
\end{align*}
The final conclusion follows from taking $\beta = C^{-1} q^{-1/2} t^{-1} \alpha$.
\end{proof}

The proof of Lemma \ref{lem1} relies on the following moment inequality for suprema of empirical processes, which is an extension of \cite{VW11}, Theorem 2.1, to possibly unbounded classes of functions (Theorem 3.1 of \cite{VW11} derives a moment inequality applicable to the case where the envelope $F$ has $q > 4$ moments, but the form of the inequality in Theorem \ref{vdVW} is more convenient in our applications; note that Theorem \ref{vdVW} only requires $F \in \mL^{2}( P )$, as opposed to $F \in \mL^{q}( P )$ with $q > 4$ in Theorem 3.1 of \cite{VW11}, and Theorem \ref{vdVW} is not covered by \cite{VW11}). Recall the uniform entropy integral $J(\delta,\mF,F)$.

\begin{theorem}[A useful maximal inequality]
\label{vdVW}
Suppose that $F \in \mL^{2} ( P )$.  Let $\delta = \sigma/\| F \|_{P,2}$. Then
 \begin{equation*}
\bE [ \| \bG_{n} \|_{\mF} ]
\lesssim J(\delta,\mF,F) \| F \|_{P,2} + \frac{\| M \|_{2} J^{2} (\delta,\mF,F)}{\delta^{2} \sqrt{n} }.
 \end{equation*}
 \end{theorem}

 In the Supplemental Material \cite{CCK14},  we give a full proof of Theorem \ref{vdVW} for the sake of completeness, although the proof is essentially similar to the proof of Theorem 2.1 in \cite{VW11}.

The bound in Theorem \ref{vdVW} will be explicit as soon as a suitable bound on the covering number is available. For example,
the following corollary is an extension of \cite{GG01}, Proposition 2.1. For its proof, see Appendix \ref{sec: proof of cor: maximal}.

\begin{corollary}[Maximal inequality specialized to VC type classes]
\label{cor: maximal}
Consider the same setup as in Theorem \ref{vdVW}. Suppose that there exist constants $A \geq e$ and $v \geq 1$ such that $\sup_{Q} N(\mF,e_{Q},\varepsilon \| F \|_{Q,2}) \leq (A/\varepsilon)^{v}, \ 0 < \forall \varepsilon \leq 1$.
Then
\begin{equation*}
\bE [ \| \bG_{n} \|_{\mF} ] \lesssim \sqrt{v\sigma^{2} \log \left ( \frac{A \| F \|_{P,2}}{\sigma} \right ) } + \frac{v\| M \|_{2}}{\sqrt{n}} \log \left ( \frac{A \| F \|_{P,2}}{\sigma} \right ).
\end{equation*}
\end{corollary}

\section{Proof of Theorem \ref{main}}

\label{proof}

We make use of Lemma \ref{strassen} to prove the theorem.
Construct a tight Gaussian random variable $G_{P}$ in $\ell^{\infty}(\mF)$ given in assumption (A3), independent of $X_{1},\dots,X_{n}$.
We note that one can extend $G_{P}$ to the linear hull of $\mF$ in such a way that $G_{P}$ has linear sample paths \citep[see][Theorem 3.1.1]{D99}.
Let $\{ f_{1},\dots,f_{N} \}$ be a minimal $\varepsilon \| F \|_{P,2}$-net of $(\mF,e_{P})$  with $N=N(\mF,e_{P},\varepsilon \| F \|_{P,2})$.
Then for every $f \in \mF$, there exists a function $f_{j}, 1 \leq j \leq N$ such that $e_{P}(f,f_{j}) < \varepsilon \| F \|_{P,2}$. Recall $\mF_{\varepsilon} = \{ f-g : f,g \in \mF, e_{P}(f,g) < \varepsilon \| F \|_{P,2} \}$ and define
\begin{equation*}
Z^{\varepsilon} = \max_{1 \leq j \leq N} \bG_{n} f_{j}, \ \widetilde{Z}^{*} = \sup_{f \in \mF} G_{P}f, \ \widetilde{Z}^{*\varepsilon} = \max_{1 \leq j \leq N} G_{P} f_{j}.
\end{equation*}
Observe that $| Z - Z^{\varepsilon} | \leq \| \bG_{n} \|_{\mF_{\varepsilon}}$ and $| \widetilde{Z}^{*\varepsilon} - \widetilde{Z}^{*} | \leq \| G_{P} \|_{\mF_{\varepsilon}}$.

We shall apply Corollary \ref{cor:Yurinskii} to $Z^{\varepsilon}$. Recall that $\log (N \vee n) = H_{n}(\varepsilon)$. Then for every Borel subset $A$ of $\R$ and $\delta > 0$,
\[
\bP ( Z^{\varepsilon} \in A ) - \bP ( \widetilde{Z}^{*\varepsilon} \in A^{16 \delta} ) \lesssim \delta^{-2} \{ B_{1} + \delta^{-1}( B_{2}  + B_{4} )  H_{n}(\varepsilon) \} H_{n}(\varepsilon) + n^{-1} \log n,
\]
where
\begin{align*}
B_{1} &= n^{-1} \bE \left [ \max_{1 \leq j,k \leq N} | \sum_{i=1}^{n}( f_{j}(X_{i})f_{k}(X_{i}) - P(f_{j} f_{k}) ) |\right ], \\
B_{2} &= n^{-3/2} \bE \left [ \max_{1 \leq j \leq N} \sum_{i=1}^{n} | f_{j}(X_{i}) |^{3} \right ], \\
B_{4} &=  n^{-1/2} \bE \left [ \max_{1 \leq j \leq N}  | f_{j}(X_{1}) |^{3} \cdot 1 \left (\max_{1 \leq j \leq N} | f_{j}(X_{1}) | > \delta \sqrt{n}H_{n}(\varepsilon)^{-1} \right) \right ].
\end{align*}
Clearly $B_{1} \leq n^{-1/2} \bE [ \| \bG_{n} \|_{\mF \cdot \mF} ], B_{2} \leq n^{-1/2}\kappa^{3}$, and $B_{4} \leq n^{-1/2} P[ F^{3} 1(F > \delta \sqrt{n}H_{n}(\varepsilon)^{-1}) ]$.
Hence choosing $\delta > 0$ in such a way that
\begin{equation*}
C  \delta^{-2}n^{-1/2} \bE [ \| \bG_{n} \|_{\mF \cdot \mF} ] H _{n}(\varepsilon) \leq \frac{\gamma}{4}, \ C \delta^{-3} n^{-1/2} \kappa^{3}  H^{2} _{n}(\varepsilon) \leq \frac{\gamma}{4},
\end{equation*}
that is,
\begin{equation*}
\delta \geq C \max \left \{ \gamma^{-1/2} n^{-1/4} (\bE [ \| \bG_{n} \|_{\mF \cdot \mF} ])^{1/2} H^{1/2}_{n}(\varepsilon),  \gamma^{-1/3}n^{-1/6} \kappa  H^{2/3} _{n}(\varepsilon) \right \},
\end{equation*}
we have
\begin{equation*}
\bP ( Z^{\varepsilon} \in A ) \leq \bP ( \widetilde{Z}^{*\varepsilon} \in A^{16 \delta} ) + \frac{\gamma}{2} + \frac{\gamma}{4} \kappa^{-3} P[ F^{3} 1(F > \delta \sqrt{n}H_{n}(\varepsilon)^{-1}) ] + \frac{C \log n}{n}.
\end{equation*}
Note that $\delta \geq c \gamma^{-1/3}n^{-1/6} \kappa  H^{2/3} _{n}(\varepsilon)$,
so that
\begin{equation*}
P[ F^{3} 1(F > \delta \sqrt{n}H_{n}(\varepsilon)^{-1}) ] \leq P[ F^{3} 1(F/\kappa > c \gamma^{-1/3} n^{1/3} H_{n}(\varepsilon)^{-1/3}) ].
\end{equation*}
Hence
\begin{align}
\bP ( Z^{\varepsilon} \in A ) &\leq \bP ( \widetilde{Z}^{*\varepsilon} \in A^{16 \delta} ) + \frac{\gamma}{2} \notag \\
&\quad + \frac{\gamma}{4} P[ (F/\kappa)^{3} 1(F/\kappa > c \gamma^{-1/3} n^{1/3} H_{n}(\varepsilon)^{-1/3}) ]+ \frac{C \log n}{n} \notag \\
&=: \bP ( \widetilde{Z}^{*\varepsilon} \in A^{16 \delta} ) + \frac{\gamma}{2} + \text{error}. \label{step-a}
\end{align}

By Theorem \ref{concentration}, with probability $> 1-\gamma/4$,
\begin{multline}
\| \bG_{n} \|_{\mF_{\varepsilon}}
\leq K(q) \big \{ \phi_{n}(\varepsilon) + (\varepsilon \| F \|_{P,2} + n^{-1/2} \| M \|_{q})\gamma^{-1/q}  \\
+ n^{-1/2} \| M \|_{2} \gamma^{-2/q} \big\}
=:a, \label{step-b}
\end{multline}
where $K(q)$ is a constant that depends only on $q$. Moreover, by the Borell-Sudakov-Tsirel'son inequality \citep[][Proposition A.1]{VW96}, with probability $> 1-\gamma/4$, we have
\begin{equation}
\| G_{P} \|_{\mF_{\varepsilon}} \leq \phi_{n}(\varepsilon)+ \varepsilon \| F \|_{P,2}\sqrt{2 \log (4/\gamma)} =: b. \label{step-c}
\end{equation}

Therefore, for every Borel subset $A$ of $\R$,
\begin{align*}
\bP (Z \in A ) &\leq \bP (Z^{\varepsilon} \in A^{a} ) + \frac{\gamma}{4} \quad (\text{by (\ref{step-b})}) \\
&\leq \bP (\widetilde{Z}^{*\varepsilon} \in A^{a+16\delta} ) + \frac{3}{4} \gamma + \text{error} \quad (\text{by (\ref{step-a})}) \\
&\leq \bP (\widetilde{Z}^{*} \in A^{a+b+16\delta} ) + \gamma + \text{error}. \quad (\text{by (\ref{step-c})})
\end{align*}
The conclusion follows from Lemma \ref{strassen}.
\qed

 \section*{Acknowledgments}
The authors would like to thank the editors and anonymous referees for their careful review that helped improve upon the quality of the paper.

\begin{supplement}
\stitle{Supplement to ``Gaussian approximation of suprema of empirical processes''}
\slink[url]{}
\sdescription{This supplemental file contains the additional technical proofs omitted in the main text, and some technical tools used in the proofs. }
\end{supplement}

\newpage

\begin{center}
{\bf \Large Supplement to ``Gaussian approximation of suprema of empirical processes''}

\medskip 
{\large By Victor Chernozhukov, Denis Chetverikov, and Kengo Kato} 

\medskip

{\em \large MIT, UCLA, and University of Tokyo}
\end{center}

\appendix

\section{Additional proofs}

\subsection{Proof of Lemma \ref{lem0}}
We first note that by approximation \citep[see][Problem 2.5.1]{VW96}, assumption (A4) implies that
\[
\int_{0}^{1} \sqrt{\log N(\mF, e_{P}, \varepsilon \| F \|_{P,2})} d \varepsilon < \infty.
\]
Let $G_{P}$ be a centered Gaussian process indexed by $\mF$ with covariance function $\bE [ G_{P}(f) G_{P}(g) ] = P(fg)$.
Recall that $\mF$ is $P$-centered, and by Example 1.3.10 in \cite{ VW96}, $\mF$ is $P$-pre-Gaussian  if and only if $(\mF,e_{P})$ is totally bounded and $G_{P}$ has a version that has sample paths almost surely uniformly $e_{P}$-continuous.   Dudley's criterion for sample continuity of Gaussian processes states that when
\begin{equation}
\int_{0}^{\infty} \sqrt{\log N(\mF,e_{P},\varepsilon)} d\varepsilon < \infty, \label{dudley-entropy}
\end{equation}
there exists a version of $G_{P}$ that has sample paths uniformly $e_{P}$-continuous \citep[][p.100-101]{VW96} (note that (\ref{dudley-entropy}) implies that $N(\mF,e_{P},\varepsilon)$ is finite for every $\varepsilon > 0$, that is,  $\mF$ is totally bounded for $e_{P}$). The lemma readily follows from these observations. \qed

\subsection{Proofs of Lemmas \ref{lem: Kolmogorov} and \ref{lem: Kolmogorov2}}

\begin{proof}[Proof of Lemma \ref{lem: Kolmogorov}]
The proof of Lemma \ref{lem: Kolmogorov} depends on the following lemma on {\em anti-concentration} of suprema of Gaussian processes.

\begin{lemma}[An anti-concentration inequality]\label{lem: anticoncentration}
Let $(S,\mS,P)$ be a probability space, and let $\mF \subset \mL^{2}( P )$ be a $P$-pre-Gaussian class of functions.
Denote by $G_{P}$ a tight Gaussian random variable in $\ell^{\infty}(\mF)$ with mean zero and covariance function $\bE[ G_{P}(f) G_{P}(g) ] = \Cov_{P}(f,g)$ for all $f,g \in \mF$ where $\Cov_{P}(\cdot,\cdot)$ denotes the covariance under $P$. Suppose that there exist constants  $\underline{\sigma}, \bar{\sigma}>0$ such that $\underline{\sigma}^{2} \leq \Var_{P}(f) \leq \bar{\sigma}^{2}$ for all $f\in\mathcal{F}$.
Then for every $\epsilon > 0$,
\[
\sup_{x\in\mathbb{R}}\mathbb{P}\left\{\left|\sup_{f\in\mathcal{F}}G_Pf-x\right|\leq \epsilon\right\}\leq C_{\sigma}\epsilon\left\{\bE \left[\sup_{f\in\mathcal{F}}G_Pf\right]+\sqrt{1\vee \log(\underline{\sigma}/\epsilon)}\right\},
\]
where $C_{\sigma}$ is a constant depending only on $\underline{\sigma}$ and $\bar{\sigma}$.
\end{lemma}

\begin{proof}[Proof of Lemma \ref{lem: anticoncentration}]
The proof of this lemma is the same as that of Theorem 2.1 in \cite{CCK13} with the exception that we now apply Theorem 3, part (ii) instead of Theorem 3, part (i) from \cite{CCK12b}.
\end{proof}

Going back to the proof of Lemma \ref{lem: Kolmogorov}, for every $t \in \R$, we have
\begin{align*}
\bP ( Z \leq t ) &= \bP ( \{ Z \leq t \} \cap \{ | Z - \widetilde{Z} | \leq r_{1} \} )+ \bP ( \{ Z \leq t \} \cap \{ | Z - \widetilde{Z} | > r_{1} \} ) \\
&\leq \bP (   \widetilde{Z} \leq t + r_{1} ) + r_{2} \\
&\leq \bP (   \widetilde{Z} \leq t ) + C_{\sigma}r_{1} \{ \bE[ \widetilde{Z} ] +\sqrt{1\vee \log(\underline{\sigma}/r_{1})} \} + r_{2},
\end{align*}
where we have used Lemma \ref{lem: anticoncentration} to deduce the last inequality.  A similar argument leads to the reverse inequality. This completes the proof. 
\end{proof}

\begin{proof}[Proof of Lemma \ref{lem: Kolmogorov2}]

Take $\beta_{n} \to \infty$ sufficiently slowly such that $\beta_{n} r_{n} (1 \vee \bE[ \widetilde{Z}_{n} ]) = o(1)$. Then since $\bP(| Z_{n} - \widetilde{Z}_{n} | > \beta_{n} r_{n}) = o(1)$, by Lemma \ref{lem: Kolmogorov}, we have
\[
\sup_{t \in \R} | \bP( Z_{n} \leq t ) - \bP (\widetilde{Z}_{n} \leq t) | = O\{ r_{n} (\bE[ \widetilde{Z}_{n} ] + | \log (\beta_{n}r_{n}) |) \} + o(1) = o(1).
\]
This completes the proof. 
\end{proof}

\subsection{Proof of Lemma \ref{strassen}}
The ``only if'' part is trivial, and hence we prove the ``if'' part.
By Strassen's theorem \citep[see][Section 10.3]{P02}, there are random variables $V^{*}$ and $W^{*}$ with distributions $\mu$ and $\nu$ such that  $\bP ( | V^{*}-W^{*} | > \delta ) \leq \varepsilon$.
$V^{*}$ may be different from $V$. Let $F(w \mid v)$ be a regular conditional distribution function of $W^{*}$ given $V^{*}=v$. Denote by $F^{-1}(\tau \mid v)$ the quantile function of $F(w \mid v)$, that is, 
$F^{-1}( \tau  \mid v ) = \inf \{ w : F( w \mid v ) \geq \tau \}$. Generate a uniform random variable $U$ on $(0,1)$ independent of $V$ and take $W(\omega) = F^{-1}(U(\omega) \mid V(\omega))$. Then it is routine to verify that
$(V,W) \stackrel{d}{=} (V^{*},W^{*})$.
\qed

\subsection{Proof of Theorem \ref{vdVW}}

We first prove the following technical lemma.
\begin{lemma}
\label{convexity}
Write $J(\delta)$ for $J(\delta,\mF,F)$ and suppose that $J(1)$ is finite (and hence $J(\delta)$ is finite for all $\delta$). Then
(i) the map $\delta \mapsto J(\delta)$ is concave; (ii) $J(c \delta) \leq c J(\delta), \ \forall c \geq 1$; (iii) the map $\delta \mapsto J(\delta)/\delta$ is non-increasing; (iv) the map $[0,\infty) \times (0,\infty) \ni (x,y) \mapsto J(\sqrt{x/y}) \sqrt{y}$ is concave.
\end{lemma}
\begin{proof}
Let $\lambda (\varepsilon) = \sup_{Q}\sqrt{1+\log N(\mF,e_{Q}, \varepsilon \| F \|_{Q,2} )}$.
Part (i) follows from the fact that the map $\varepsilon \mapsto \lambda(\varepsilon)$ is non-increasing. Part (ii) follows from the inequality
\begin{equation*}
 \int_{0}^{c\delta} \lambda(\varepsilon) d\varepsilon =  c \int_{0}^{\delta}  \lambda(c\varepsilon) d\varepsilon \leq c \int_{0}^{\delta}  \lambda(\varepsilon) d\varepsilon.
\end{equation*}
Part (iii) follows from the identity
\begin{equation*}
\frac{J(\delta)}{\delta} = \int_{0}^{1} \lambda (\delta \varepsilon) d\varepsilon.
\end{equation*}
The proof of part (iv) uses some facts in convex analysis. Proofs of the following lemmas can be found in, for example, \cite{BV04}, Section 3.2.

\begin{lemma}
Let $D$ be a convex subset of $\R^{n}$, and let $f: D \to \R$ be a concave function.
Then the {\em perspective} $(x,t) \mapsto t f(x/t), \ \{ (x,t) \in \R^{n+1} : x/t \in D, t > 0 \} \to \R$,  is also concave.
\end{lemma}

\begin{lemma}
Let $D_{1}$ be a convex subset of $\R^{n}$, and let $g_{i}: D_{1} \to \R, 1 \leq i \leq k$ be concave functions. Let $D_{2}$ denote the convex hull of the set $\{ (g_{1}(x),\dots,g_{k}(x)) : x \in D_{1} \}$. Let $h: D_{2} \to \R$ be concave and nondecreasing in each coordinate. Then $f(x) = h(g_{1}(x),\dots,g_{k}(x)), D_{1} \to \R$, is concave.
\end{lemma}

Let $h(s,t) = J(s/t) t$, $g_{1}(x,y) = \sqrt{x}$ and $g_{2}(x,y) = \sqrt{y}$. Then $h$ is concave and nondecreasing in each coordinate, and $g_{i}, i =1,2$ are concave. Hence $J(\sqrt{x/y}) \sqrt{y} = h(g_{1}(x,y),g(x,y))$ is concave.
\end{proof}

We will use a version of the contraction principle for Rademacher averages.
Recall that a Rademacher random variable is a random variables taking $\pm 1$ with equal probability.

\begin{lemma}[A contraction principle, \cite{LT91}]
\label{contraction2}
Let $\varepsilon_{1},\dots,\varepsilon_{n}$ be i.i.d. Rademacher random variables independent of $X_{1},\dots,X_{n}$.
Then
\begin{equation*}
\bE \left [ \left \| \sum_{i=1}^{n} \varepsilon_{i} f^{2}(X_{i}) \right \|_{\mF} \right ] \leq 4 \bE \left [ M  \left \| \sum_{i=1}^{n} \varepsilon_{i}  f(X_{i})  \right \|_{\mF} \right ].
\end{equation*}
\end{lemma}

\begin{proof}
See \cite{LT91}, Theorem 4.12, and the discussion following the theorem.
\end{proof}

We will also use the following form of the Hoffmann-J\o rgensen inequality.
\begin{theorem}[A Hoffmann-J\o rgensen-type inequality, \cite{LT91}]
\label{thm:HJ}
Let $\varepsilon_{1},\dots,\varepsilon_{n}$ be i.i.d. Rademacher random variables independent of $X_{1},\dots,X_{n}$.
Then for every $1 < q < \infty$,
\begin{equation*}
\left ( \bE \left [ \left \| \sum_{i=1}^{n} \varepsilon_{i} f(X_{i}) \right \|_{\mF}^{q} \right ] \right )^{1/q} \leq K(q) \left [ \bE \left  [ \left \| \sum_{i=1}^{n} \varepsilon_{i} f(X_{i}) \right \|_{\mF} \right ]+ \| M \|_{q} \right ],
\end{equation*}
where $K(q)$ is a constant depending only on $q$. 
\end{theorem}

\begin{proof}
See, for example, \cite{LT91}, Theorem 6.20.
\end{proof}

We are now in position to prove Theorem \ref{vdVW}.

\begin{proof}[Proof of Theorem \ref{vdVW}]
We may assume that $J(1)$ is finite since otherwise $J(\delta)$ is infinite and there is nothing to prove.
Moreover, without loss of generality, we may assume that $F$ is everywhere positive.
Let $P_{n}$ denote the empirical distribution that assigns probability $n^{-1}$ to each $X_{i}$.
Let $\sigma_{n}^{2} = \sup_{f \in \mF} n^{-1} \sum_{i=1}^{n} f^{2}(X_{i})$. For i.i.d. Rademacher random variables $\varepsilon_{1},\dots,\varepsilon_{n}$ independent of $X_{1},\dots,X_{n}$, the symmetrization inequality gives
\begin{equation*}
\bE [ \| \bG_{n} \|_{\mF} ] \leq 2 \bE \left  [ \left \|  \frac{1}{\sqrt{n}} \sum_{i=1}^{n} \varepsilon_{i} f(X_{i}) \right \|_{\mF}  \right ].
\end{equation*}
Here the standard entropy integral inequality gives
\begin{align*}
&\bE \left  [ \left \|  \frac{1}{\sqrt{n}} \sum_{i=1}^{n} \varepsilon_{i} f(X_{i}) \right \|_{\mF}  \mid X_{1},\dots,X_{n} \right ]
\leq C  \int_{0}^{\sigma_{n}} \sqrt{1 + \log N(\mF, e_{P_{n}},\varepsilon)} d \varepsilon \\
&\quad \leq C   \| F \|_{P_{n},2}  \int_{0}^{\sigma_{n}/\| F \|_{P_{n},2}} \sqrt{1 + \log N(\mF, e_{P_{n}},\varepsilon \| F \|_{P_{n},2} )} d \varepsilon \\
&\quad \leq C \| F \|_{P_{n},2} J(\sigma_{n}/\| F \|_{P_{n},2}).
\end{align*}
Hence by Lemma \ref{convexity} (iv) and Jensen's inequality,
\begin{equation*}
Z := \bE \left  [ \left \|  \frac{1}{\sqrt{n}} \sum_{i=1}^{n} \varepsilon_{i} f(X_{i}) \right \|_{\mF}  \right ]  \leq C\| F \|_{P,2} J(\sqrt{\bE[ \sigma_{n}^{2} ]}/\| F \|_{P,2}).
\end{equation*}
By the symmetrization inequality, the contraction principle (Lemma \ref{contraction2}) and the Cauchy-Schwarz inequality,
\begin{align*}
&\bE [ \sigma_{n}^{2} ] \leq \sigma^{2} + \bE\left [  \left \| \bEn[  (f^{2}(X_{i}) - Pf^{2}) ] \right \|_{\mF} \right ] \leq \sigma^{2} + 2 \bE\left [  \left \| \bEn[ \varepsilon_{i}  f^{2}(X_{i})]  \right \|_{\mF} \right ] \\
&\quad \leq \sigma^{2} + 8 \bE\left [ M  \left \| \bEn[ \varepsilon_{i} f(X_{i}) ] \right \|_{\mF} \right ] \leq  \sigma^{2} + 8 \| M \|_{2} \left ( \bE \left [ \left \| \bEn[ \varepsilon_{i} f(X_{i}) ] \right \|^{2}_{\mF}  \right ] \right )^{1/2}.
\end{align*}
Here by the Hoffmann-J\o rgensen inequality (Theorem \ref{thm:HJ}),
\begin{equation*}
\left ( \bE \left [ \left \| \bEn[ \varepsilon_{i} f(X_{i}) ] \right \|^{2}_{\mF}  \right ] \right )^{1/2} \lesssim \bE \left [ \left \| \bEn[ \varepsilon_{i} f(X_{i}) ] \right \|_{\mF} \right ]+ n^{-1} \| M \|_{2},
\end{equation*}
so that,
\begin{equation*}
\sqrt{\bE[ \sigma_{n}^{2} ]} \leq C\| F \|_{P,2}  ( \Delta  \vee \sqrt{D Z} ),
\end{equation*}
where $\Delta^{2} := \max \{ \sigma^{2},n^{-1} \| M \|_{2}^{2}  \}/ \| F \|^{2}_{P,2} \geq \delta^{2}$ and $D:= \| M \|_{2}/ ( \sqrt{n} \| F \|_{P,2}^{2})$.
Therefore, using Lemma \ref{convexity} (ii), we have
\begin{equation*}
Z \leq C\| F \|_{P,2} J(\Delta \vee \sqrt{D Z})
\end{equation*}
We consider the following two cases:

(i) $\sqrt{DZ} \leq \Delta$. In this case, $J(\Delta \vee \sqrt{D Z}) \leq J(\Delta)$, so that $Z \leq C \| F \|_{P,2} J(\Delta)$.
Since the map $\delta \mapsto J(\delta)/\delta$ is non-increasing (Lemma \ref{convexity} (iii)),
\begin{equation*}
J(\Delta)= \Delta \frac{J(\Delta)}{\Delta} \leq \Delta \frac{J(\delta)}{\delta} = \max \left \{ J(\delta), \frac{\| M \|_{2}J(\delta)}{\sqrt{n}\delta \| F \|_{P,2}} \right \}.
\end{equation*}
Since $J(\delta)/\delta \geq J(1) \geq 1$, the last expression is bounded by
\begin{equation*}
\max \left \{ J(\delta), \frac{\| M \|_{2}J^{2}(\delta)}{\sqrt{n}\delta^{2} \| F \|_{P,2}} \right \}.
\end{equation*}

(ii) $\sqrt{DZ} \geq \Delta$. In this case, $J(\Delta \vee \sqrt{D Z}) \leq J(\sqrt{D Z})$, and since the map $\delta \mapsto J(\delta)/\delta$ is non-increasing (Lemma \ref{convexity} (iii)),
\begin{equation*}
 J(\sqrt{D Z}) = \sqrt{DZ}  \frac{J(\sqrt{D Z})}{\sqrt{DZ}} \leq \sqrt{DZ} \frac{J(\Delta)}{\Delta} \leq \sqrt{DZ} \frac{J(\delta)}{\delta}.
 \end{equation*}
 Therefore,
 \begin{equation*}
 Z \leq C \| F \|_{P,2} \sqrt{DZ}  \frac{J(\delta)}{\delta},
 \end{equation*}
 that is
 \begin{align*}
Z \leq C  \| F \|^{2}_{P,2} D  \frac{J^{2}(\delta)}{\delta^{2}} =  \frac{C \| M \|_{2} J^{2}(\delta)}{\sqrt{n}\delta^{2}}.
 \end{align*}
 This completes the proof.
\end{proof}


\subsection{Proof of Corollary \ref{cor: maximal}}
\label{sec: proof of cor: maximal}

Observe that
\begin{equation*}
J(\delta) \leq \int_{0}^{\delta} \sqrt{1+v \log (A/\varepsilon)} d\varepsilon \leq A \sqrt{v}  \int_{A/\delta}^{\infty} \frac{\sqrt{1+\log \varepsilon}}{\varepsilon^{2}} d\varepsilon.
\end{equation*}
An integration by parts gives
\begin{align*}
\int_{c}^{\infty} \frac{\sqrt{1+\log \varepsilon}}{\varepsilon^{2}} d\varepsilon &= \left [ - \frac{\sqrt{1+\log \varepsilon}}{\varepsilon} \right ]_{c}^{\infty} + \frac{1}{2} \int_{c}^{\infty} \frac{1}{\varepsilon^{2} \sqrt{1+\log \varepsilon}} d \varepsilon \\
&\leq \frac{\sqrt{1+\log c}}{c} + \frac{1}{2} \int_{c}^{\infty} \frac{ \sqrt{1+\log \varepsilon}}{\varepsilon^{2}} d \varepsilon, \ \text{if} \ c \geq e.
\end{align*}
by which we have
\begin{equation*}
\int_{c}^{\infty} \frac{\sqrt{1+\log \varepsilon}}{\varepsilon^{2}} d\varepsilon \leq \frac{2\sqrt{1+\log c}}{c} \leq  \frac{2\sqrt{2}\sqrt{\log c}}{c}, \ \text{if} \ c \geq e,
\end{equation*}
Since $A/\delta \geq A \geq e$, we have
\begin{equation*}
J(\delta) \leq 2\sqrt{2v} \delta \sqrt{\log (A/\delta)}.
\end{equation*}
Applying Theorem \ref{vdVW}, we obtain the desired conclusion.
\qed

\subsection{Proof of Lemma \ref{lem1}}

Before proving Lemma \ref{lem1}, we shall recall the following lemma on uniform entropy numbers. 

\begin{lemma}
\label{lem: uniform cov}
Let $\mF_{1},\dots,\mF_{k}$ be classes of measurable functions $S \to \R$ to which measurable envelopes $F_{1},\dots,F_{k}$ are attached, respectively, and let $\phi: \R^{k} \to \R$ be a map that is Lipschitz  in the sense that 
\[
| \phi \circ f(x) - \phi \circ g(x) |^{2} \leq \sum_{j=1}^{k} L^{2}_{j} (x) | f_{j} (x) - g_{j}(x) |^{2},
\]
for every $f = (f_{1},\dots,f_{k}), g=(g_{1},\dots,g_{k}) \in \mF_{1} \times \cdots \times \mF_{k} =: \mF$ and every $x \in S$, where $L_{1},\dots,L_{k}$ are non-negative measurable functions on $S$.   Consider the class of functions $\phi (\mF) := \{ \phi \circ f : f \in \mF \}$. 
Denote $(\sum_{j=1}^{k} L_{j}^{2} F_{j}^{2})^{1/2}$ by $L \cdot F$.  Then we have 
\[
\sup_{Q} N(\phi (\mF), e_{Q}, \varepsilon \|  L \cdot F \|_{Q,2}) \leq \prod_{j=1}^{k} \sup_{R_{j}} N(\mF_{j}, e_{R_{j}}, \varepsilon \| F_{j} \|_{R_{j},2}),
\]
for every $0 < \varepsilon \leq 1$, where the suprema are taken over all finitely discrete probability measures on $(S,\mS)$.
\end{lemma}

\begin{proof}[Proof of Lemma \ref{lem: uniform cov}]
The proof is implicit in \cite{VW96}, p.199, and hence omitted. 
\end{proof}

We will use the following corollary to the above lemma. 

\begin{corollary}
\label{covnumber2}
(i) Let $\mF$ and $\mG$ be classes of measurable functions $S \to \R$, to which measurable envelopes $F$ and $G$ are attached, respectively.
Denote by $\mF \cdot \mG$ the pointwise product of $\mF$ and $\mG$. Then
\begin{align*}
&\sup_{Q} N(\mF \cdot \mG, e_{Q}, \sqrt{2} \varepsilon \| F G \|_{Q,2}) \\
&\quad \leq \sup_{Q} N(\mF, e_{Q}, \varepsilon \| F \|_{Q,2})\sup_{Q} N(\mG, e_{Q}, \varepsilon \|  G \|_{Q,2}),
\end{align*}
for every $0 < \varepsilon \leq 1$, where the suprema are taken over all finitely discrete probability measures $Q$ on $(S,\mS)$.

(ii) Let $\mF$ be a class of measurable functions $S \to \R$, to which a measurable envelope $F$ is attached. For every $q \geq 1$, let $\mF (q) = \{ | f |^{q} : f \in \mF \}$.
Then
\begin{equation*}
\sup_{Q} N(\mF (q),e_{Q}, q\varepsilon \| F^{q} \|_{Q,2}) \leq \sup_{Q} N(\mF,e_{Q},\varepsilon \| F \|_{Q,2}),
\end{equation*}
for every $0 < \varepsilon \leq 1$, where the suprema are taken over all finitely discrete probability measures $Q$ on $(S,\mS)$.
\end{corollary}

\begin{proof}[Proof of Corollary \ref{covnumber2}]
(i) Take $k=2, \mF_{1} = \mF, F_{1} = F, \mF_{2} = \mG, F_{2} = G$, and $\phi: \R^{2} \to \R$ as $\phi (s,t) = st$. Then we can take $L_{1} = F , L_{2} = G$, and the desired conclusion directly follows from Lemma \ref{lem: uniform cov}. 

(ii) This follows from application of Lemma \ref{lem: uniform cov} with $k = 1$ and  $\phi (s) = |s|^{q}$. 
\end{proof}

\begin{proof}[Proof of Lemma \ref{lem1}]
For the first inequality, noting that $J(\delta,\mF_{\varepsilon},2F) \lesssim J(\delta,\mF,F) = J(\delta)$,  by Theorem \ref{vdVW}, we have
\[
\bE[ \| \bG_{n} \|_{\mF_{\varepsilon}} ] \lesssim J(\varepsilon) \| F \|_{P,2} + n^{-1/2} \varepsilon^{-2} J^{2}(\varepsilon) \| M \|_{2}.
\]
Moreover, by Dudley's inequality \citep[][Corollary 2.2.8]{VW96},  $\bE [ \| G_{P} \|_{\mF_{\varepsilon}}  ]  \lesssim J(\varepsilon) \| F \|_{P,2}$. Note that by approximation \citep[see][Problem 2.5.1]{VW96}, we have $$\int_{0}^{\delta} \sqrt{1+\log N(\mF,e_{P}, \tau \| F \|_{P,2})} d \tau \lesssim J(\delta).$$
Hence the first inequality is proved.

The third inequality is deduced from Theorem \ref{vdVW} together with the covering number estimate,
\[
\sup_{Q} N(\mF \cdot \mF, e_{Q}, \sqrt{2} \varepsilon \| F^{2} \|_{Q,2}) \leq \sup_{Q} N^{2}(\mF, e_{Q}, \varepsilon \| F \|_{Q,2}), 
\]
which follows from Corollary \ref{covnumber2} (i). 
Hence we shall prove the second inequality. We first observe that
\begin{equation*}
\bEn[| f(X_{i}) |^{3} ] = P | f |^{3} + n^{-1/2} \bG_{n} (| f |^{3}),
\end{equation*}
by which we have
\begin{equation*}
\bE\left  [ \| \bEn[ | f(X_{i}) |^{3} ] \|_{\mF}  \right ] \leq \sup_{f \in \mF} P|f|^{3} + n^{-1/2} \bE [ \| \bG_{n} (| f |^{3}) \|_{\mF} ].
\end{equation*}
Let $\varepsilon_{1},\dots,\varepsilon_{n}$ be i.i.d. Rademacher random variables independent of $X_{1},\dots,X_{n}$.
By the symmetrization inequality,
\begin{equation*}
\bE [ \| \bG_{n} (| f |^{3}) \|_{\mF} ] \leq 2\bE \left [ \left \|  \frac{1}{\sqrt{n}} \sum_{i=1}^{n} \varepsilon_{i}  | f(X_{i}) |^{3} \right \|_{\mF} \right ].
\end{equation*}
By the contraction principle together with the Cauchy-Schwarz inequality,
\begin{align*}
\bE \left [ \left \| \sum_{i=1}^{n} \varepsilon_{i}  | f(X_{i}) |^{3} \right \|_{\mF} \right ]
&\lesssim \bE \left [ M^{3/2}\left \| \sum_{i=1}^{n} \varepsilon_{i}  | f(X_{i}) |^{3/2} \right \|_{\mF} \right] \\
&\leq \| M \|_{3}^{3/2} \left( \bE \left [ \left \|  \sum_{i=1}^{n} \varepsilon_{i}  | f(X_{i}) |^{3/2} \right \|^{2}_{\mF} \right] \right )^{1/2}.
\end{align*}
Moreover, by the Hoffmann-J\o rgensen inequality,
\begin{equation*}
\left( \bE \left [ \left \|  \sum_{i=1}^{n} \varepsilon_{i}  | f(X_{i}) |^{3/2} \right \|^{2}_{\mF} \right] \right )^{1/2} 
\lesssim \bE \left [ \left \|   \sum_{i=1}^{n} \varepsilon_{i}  | f(X_{i}) |^{3/2} \right \|_{\mF} \right] + \| M \|_{3}^{3/2}.
\end{equation*}
By Theorem \ref{vdVW} together with Corollary \ref{covnumber2} (ii), we have
\begin{align*}
\bE \left [ \left \|  \frac{1}{\sqrt{n}} \sum_{i=1}^{n} \varepsilon_{i}  | f(X_{i}) |^{3/2} \right \|_{\mF} \right]
&\lesssim J(\delta_{3}^{3/2},\mF,F) \| F^{3/2} \|_{P,2} \\
&\quad + \frac{\| M^{3/2} \|_{2} J^{2}(\delta_{3}^{3/2},\mF,F)}{\sqrt{n} \delta_{3}^{3}},
\end{align*}
by which we have
\begin{multline*}
\bE\left  [ \| \bEn[ | f(X_{i}) |^{3} ] \|_{\mF}  \right ] - \sup_{f \in \mF} P |f|^{3} \lesssim n^{-1} \| M \|^{3}_{3} \\
 + n^{-1/2} \| M \|_{3}^{3/2} \left [ J(\delta_{3}^{3/2},\mF,F) \| F \|^{3/2}_{P,3}  + \frac{\| M \|^{3/2}_{3} J^{2}(\delta_{3}^{3/2},\mF,F)}{\sqrt{n} \delta_{3}^{3}} \right ].
\end{multline*}
A further simplification is possible.
By Lemma \ref{convexity} (iii), the map $\delta \mapsto J(\delta,\mF,F)/\delta$ is non-increasing, so that $J^{2}(\delta_{3}^{3/2},\mF,F)/\delta_{3}^{3} \geq J^{2}(1,\mF,F) \geq 1$.
Hence the first term on the right side is not larger than
\[
\| M \|_{3}^{3}J^{2}(\delta_{3}^{3/2},\mF,F)/(n\delta_{3}^{3}).
\]
This completes the proof.
\end{proof}

\subsection{Proof of Corollary \ref{cor: main corollary from theorem}}
The proof consists of applying Theorem \ref{main}. Standard calculations show that for any $\delta\in(0,1)$,
\[
J(\delta):=\int_0^\delta\sup_Q\sqrt{1+\log N(\mathcal{F},L_2(Q),\varepsilon\|F\|_{Q,2})}d\varepsilon\lesssim \delta\sqrt{v\log(A/\delta)}.
\]
Further, for some sufficiently large $C$, let $\kappa_n:=C(b\sigma^2+b^3K_n n^{-1+3/q})^{1/3}$. Also note that for $k=2,3,4$, $\|M\|_{k}\leq \|M\|_q\leq n^{1/q}b$. Therefore, Lemma \ref{lem1} implies
\begin{align*}
&\mathbb{E}\left[\|\bEn[|f(X_i)|^3]\|_{\mathcal{F}}\right]- \sup_{f\in\mathcal{F}}P|f|^3\\
&\qquad \lesssim  n^{-1/2+3/(2q)}b^{3/2}\left(b^{3/2}J(\delta_3^{3/2})+\frac{b^{3/2}J^2(\delta_3^{3/2})}{n^{1/2-3/(2q)}\delta_3^3}\right),
\end{align*}
for any $\delta_3\geq \sup_{f\in\mathcal{F}}\|f\|_{P,3}/\|F\|_{P,3}$. Setting $\delta_3=b^{1/3}\sigma^{2/3}/b=(\sigma/b)^{2/3}$ gives
\begin{align*}
&\mathbb{E}\left[\|\bEn[|f(X_i)|^3]\|_{\mathcal{F}}\right]-b\sigma^2\\
&\qquad \lesssim n^{-1/2+3/(2q)}b^{3/2}\left(b^{3/2-1}\sigma K_n^{1/2}+b^{3/2}K_n n^{-1/2+3/(2q)}\right),
\end{align*}
so that using the elementary inequality $2xy\leq x^2+y^2$, we obtain
$$
\mathbb{E}\left[\|\bEn[|f(X_i)|^3]\|_{\mathcal{F}}\right] \leq C(b\sigma^2+b^3K_n n^{-1+3/q})\leq \kappa_n^3.
$$
Further, let $\varepsilon_n=\sigma/(b n^{1/2})$. Then 
$$
H_n(\varepsilon_n)=\log(N(\mathcal{F},e_P,\varepsilon_n\|F\|_{P,2})\vee n)\lesssim K_n
$$ 
and $J(\varepsilon_n)\leq C \sigma K_n^{1/2}/(b n^{1/2})$. 

Also, for $q\in[4,\infty)$, we have by the Markov inequality that
\begin{align*}
\delta_n(\varepsilon_n,\gamma)=&\frac{1}{4}P\{(F/\kappa_n)^31(F/\kappa_n>c\gamma^{-1/3}n^{1/3}H_n(\varepsilon_n)^{-1/3})\}\\
&\lesssim (b/\kappa_n)^q(\gamma^{1/3}H_n(\varepsilon_n)^{1/3}n^{-1/3})^{q-3}\\
&\lesssim (K_n^{-1/3} n^{-1/q+1/3})^q(\gamma^{1/3}H_n(\varepsilon_n)^{1/3}n^{-1/3})^{q-3}\\
&\leq \gamma^{q/3-1}/K_n\leq 1
\end{align*}
for any $\gamma\in(0,1)$.
For $q=\infty$, note that since $C$ in the definition of $\kappa_n$ is sufficiently large, $b/\kappa_n<n^{1/3}H_n(\varepsilon_n)^{-1/3}$, and so $\delta_n(\varepsilon_n,\gamma)=0$ for any $\gamma\in(0,1)$.

Now, Theorem \ref{main} combined with Lemma \ref{lem1} shows that for any $\gamma\in(0,1)$ and $\delta_4\geq \sup_{f\in\mathcal{F}}\|f\|_{P,4}/\|F\|_{P,4}$, one can construct a random variable $\widetilde{Z}$ such that $\widetilde{Z} \stackrel{d}{=} \sup_{f \in \mF} G_{P}f$ and
\begin{align}\label{eq: consequence of gaussian approximation}
\bP\left(|Z-\widetilde{Z}|>K(q)\Delta_n(\varepsilon_n,\gamma)\right)&\leq \gamma(1+\delta(\varepsilon_n,\gamma))+C(\log n)/n\\
&\lesssim \gamma+(\log n)/n, \nonumber
\end{align}
where
$K(q)$ is a constant that depends only on $q$, and
\begin{align*}
\Delta_n(\varepsilon_n,\gamma) &:= \phi_n(\varepsilon_n)+\gamma^{-1/q}\varepsilon_n b+\gamma^{-1/q}b n^{-1/2+1/q}+\gamma^{-2/q}b n^{-1/2+1/q}\\
&\quad +\gamma^{-1/2}\mathcal{E}_n^{1/2}H_n^{1/2}(\varepsilon_n)n^{-1/4}+\gamma^{-1/3}\kappa_n H_n^{2/3}(\varepsilon_n)n^{-1/6},\\
\phi_n(\varepsilon_n) &\lesssim b J(\varepsilon_n)+\varepsilon_n^{-2}b J^2(\varepsilon_n)n^{-1/2+1/q},\\
\mathcal{E}_n &:=\mathbb{E}[\|\mathbb{G}_n\|_{\mathcal{F}\cdot\mathcal{F}}]\lesssim b^2J(\delta_4^2)+\delta_4^{-4}b^2J^2(\delta_4^2)n^{-1/2+2/q}.
\end{align*}
Using the bound derived above, we have
\[
\phi_n(\varepsilon_n)\lesssim \sigma K_n^{1/2}n^{-1/2}+b K_n n^{-1/2+1/q}\lesssim b K_n n^{-1/2+1/q},
\]
and setting $\delta_4=(b^2\sigma^2)^{1/4}/b=(\sigma/b)^{1/2}$,
\[
\mathcal{E}_n\lesssim b \sigma K_n^{1/2}+b^2K_n n^{-1/2+2/q}.
\]
Also, setting $c\geq 1$ in the definition of $K_n$, so that $K_n=cv(\log n\vee \log(Ab/\sigma))\geq 1$, and using $\gamma<1$ gives
\begin{align*}
&\gamma^{-1/q}\varepsilon_n b+\gamma^{-1/q} b n^{-1/2+1/q}+\gamma^{-2/q} b n^{-1/2+1/q}\lesssim \gamma^{-1/2}b K_n n^{-1/2+1/q},\\
&\gamma^{-1/2}\mathcal{E}_n^{1/2}H_n^{1/2}(\varepsilon_n)n^{-1/4}\lesssim \gamma^{-1/2}(b\sigma)^{1/2}K_n^{3/4}n^{-1/4}+\gamma^{-1/2}b K_n n^{-1/2+1/q},\\
&\gamma^{-1/3}\kappa_n H_n^{2/3}(\varepsilon_n)n^{-1/6}\lesssim \gamma^{-1/3}b^{1/3}\sigma^{2/3}K_n^{2/3}n^{-1/6}+\gamma^{-1/3}b K_n n^{-1/2+1/q}.
\end{align*}
Substituting these bounds into (\ref{eq: consequence of gaussian approximation}) and using the definition of $\Delta_n(\varepsilon_n,\gamma_n)$, we obtain the asserted claim.
\qed

\subsection{Proofs of Propositions \ref{local:uniform}-\ref{series}}



\begin{proof}[Proof of Proposition \ref{local:uniform}]
For given $x \in \mathcal{I}, g \in \mG$ and $h > 0$, define
\begin{equation*}
f_{x,g,h}(y,t) = c_{n}(x,g) g(y) k(h^{-1}(t-x)), \ (y,t) \in \mathcal{Y} \times \R^{d}.
\end{equation*}
Consider the class of functions $\mF_{n} = \{ f_{x,g,h_{n}} - \bE[  f_{x,g,h_{n}}(Y_{1},X_{1}) ] : (x,g) \in \mathcal{I} \times \mG \}$.
We shall apply Corollary \ref{cor: main corollary from theorem} to $\mF_{n}$. Let $Z_{n} = \sup_{f \in \mF_{n}} \bG_{n} f$. We first note that   $| f_{x,g,h}(y,t) | \leq  C_{\mathcal{I} \times \mG} b \| k \|_{\infty}$ so that $|f_{x,g,h} (y,t)- \bE[  f_{x,g,h}(Y_{1},X_{1}) ] | \leq 2 C_{\mathcal{I} \times \mG} b \| k \|_{\infty} \equiv F$.
It is not difficult to see that $\mF_{n}$ is pointwise measurable.
Using Corollary \ref{covnumber2} (i), we can prove that there are constants $A,v > 0$ such that
\begin{equation}
\sup_{Q}N(\mF_{n}, e_{Q}, 2\varepsilon C_{\mathcal{I} \times \mG} b \| k \|_{\infty}) \leq (A/\varepsilon)^{v}, \ 0 < \forall \varepsilon \leq 1, \ \forall n \geq 1. \label{covnumb}
\end{equation}
Hence  for every $n \geq 1$, $\mF_{n}$ is pre-Gaussian and there exists a tight Gaussian random variable $G_{n}$ in $\ell^{\infty}(\mF_{n})$ with mean zero and covariance function
\begin{equation*}
\bE[ G_{n}(f) G_{n}(\check{f}) ] = \Cov (f(Y_{1},X_{1}), \check{f}(Y_{1},X_{1})), \ f, \check{f} \in \mF_{n}.
\end{equation*}

To apply Corollary \ref{cor: main corollary from theorem}, note that
\begin{align*}
&\bE [| f_{x,g,h_{n}}(Y_{1},X_{1}) - \bE[f_{x,g,h_{n}}(Y_{1},X_{1})]|^{3} ] \lesssim \bE [| f_{x,g,h_{n}}(Y_{1},X_{1})|^{3} ] \\
&\quad = | c_{n}(x,g) |^{3} \int_{\R^{d}} \bE[ | g (Y_{1}) |^{3} \mid X_{1}=t] | k(h_{n}^{-1}(t-x)) |^{3} p(t) dt \\
&\quad =| c_{n}(x,g) |^{3}  h_{n}^{d} \int_{\R^{d}} \bE[ | g (Y_{1}) |^{3} \mid X_{1}=x + h_{n}t] | k(t) |^{3} p(x + h_{n}t) dt \\
&\quad \leq C_{\mathcal{I} \times \mG}^{3} b^{3}  \| p \|_{\infty} h_{n}^{d} \int_{\R^{d}} | k(t) |^{3} dt,  \\
\intertext{and}
&\bE [| f_{x,g,h_{n}}(Y_{1},X_{1}) - \bE[f_{x,g,h_{n}}(Y_{1},X_{1})]|^{4} ] \lesssim \bE[| f_{x,g,h_{n}}(Y_{1},X_{1})|^{4}] \\
&\quad = | c_{n}(x,g) |^{4} h_{n}^{d} \int_{\R^{d}} \bE[ | g (Y_{1}) |^{4} \mid X_{1}=x + h_{n}t] | k(t) |^{4} p(x + h_{n}t) dt \\
&\quad \leq C_{\mathcal{I} \times \mG}^{4} b^{4}  \| p \|_{\infty} h_{n}^{d} \int_{\R^{d}} | k(t) |^{4} dt.
\end{align*}
Thus, applying Corollary \ref{cor: main corollary from theorem} with parameters $\gamma,b,\sigma$ in the corollary satisfying $\gamma=\gamma_n=(\log n)^{-1}$, $b=O(1)$ and $\sigma=\sigma_n=h_n^{d/2}$  shows that there exists a sequence $\widetilde{Z}_{n}$ of random variables such that $\widetilde{Z}_{n} \stackrel{d}{=} \sup_{f \in \mF_{n}} G_{n}f$ and as $n \to \infty$,
\begin{equation*}
| Z_{n} - \widetilde{Z}_{n} | = O_{\bP}(n^{-1/6} h_{n}^{d/3} \log n +n^{-1/4} h_{n}^{d/4} \log^{5/4} n+n^{-1/2} \log^{3/2}n ).
\end{equation*}
This implies the conclusion of the theorem. In fact, let
\begin{equation*}
B_{n}(x,g) = h_{n}^{-d/2}G_{n}(f_{x,g,h_{n}}), \ (x,g) \in \mathcal{I} \times \mG,
\end{equation*}
and $\widetilde{W}_{n} = h_{n}^{-d/2} \widetilde{Z}_{n}$.
Then $B_{n}$ is the desired Gaussian process, and as $W_{n} = h_{n}^{-d/2}Z_{n}$,
we have $\widetilde{W}_{n} \stackrel{d}{=} \sup_{(x,g) \in \mathcal{I} \times \mG} B_{n}(x,g)$ and
\begin{multline*}
| W_{n} - \widetilde{W}_{n} | = h_{n}^{-d/2} | Z_{n} - \widetilde{Z}_{n} | \\
= O_{\bP}\{ (nh_{n}^{d})^{-1/6} \log n + (nh_{n}^{d})^{-1/4} \log^{5/4} n + n^{-1/2} h_{n}^{-d/2} \log^{3/2} n \}.
\end{multline*}
This completes the proof.
\end{proof}

\begin{proof}[Proof of Proposition \ref{local:moment}]
We shall follow the notation used in the proof of Proposition \ref{local:uniform}. Take $F(y,x) = C_{\mathcal{I} \times \mG} \| k \|_{\infty} (G(y) + \bE [ G(Y_{1}) ])$ as an envelope of $\mF_{n}$.
A version of inequality (\ref{covnumb}) continues to hold with $2C_{\mathcal{I} \times \mG} b \| k \|_{\infty}$ replaced by $\| F \|_{Q,2}$. Let $D=\sup_{x \in \R^{d}} \bE [ G^{4}(Y_{1}) \mid X_{1} = x]$.
Then we have
\begin{align*}
&\bE [| f_{x,g,h_{n}}(Y_{1},X_{1}) - \bE[f_{x,g,h_{n}}(Y_{1},X_{1})] |^{3} ] \lesssim \bE [| f_{x,g,h_{n}}(Y_{1},X_{1}) |^{3} ] \\
&\quad \leq (1+D) C_{\mathcal{I} \times \mG}^{3} \| p \|_{\infty} h_{n}^{d} \int_{\R^{d}} | k(t) |^{3} dt, 
\end{align*}
and
\begin{align*}
&\bE [| f_{x,g,h_{n}}(Y_{1},X_{1}) - \bE[f_{x,g,h_{n}}(Y_{1},X_{1})]|^{4} ] \lesssim \bE [| f_{x,g,h_{n}}(Y_{1},X_{1}) |^{4} ]\\
&\quad \leq  D C_{\mathcal{I} \times \mG}^{4} \| p \|_{\infty}  h_{n}^{d} \int_{\R^{d}} | k(t) |^{4} dt.
\end{align*}
Thus, applying Corollary \ref{cor: main corollary from theorem} with parameters $\gamma,b,\sigma$ in the corollary satisfying $\gamma=\gamma_n=(\log n)^{-1}$, $b=O(1)$, and $\sigma=\sigma_n=h_n^{d/2}$ shows that there exists a sequence $\widetilde{Z}_{n}$ of random variables such that $\widetilde{Z}_{n} \stackrel{d}{=} \sup_{f \in \mF_{n}} G_{n}f$ and as $n \to \infty$,
\begin{equation*}
| Z_{n} - \widetilde{Z}_{n} | = O_{\bP}(n^{-1/6} h_{n}^{d/3} \log n +n^{-1/4} h_{n}^{d/4} \log^{5/4} n+n^{-1/2+1/q} \log^{3/2}n ).
\end{equation*}
The rest of the proof is the same as in the previous one.
\end{proof}

\begin{proof}[Proof of Proposition \ref{series}]
We only deal with case  (ii). The proof for case (i) is similar. Observe first that by condition (C2),
\begin{equation*}
| \alpha_{n}(x,g) | \leq \frac{C_{1} | \psi^{K_{n}}(x) | }{c_{1} | \psi^{K_{n}}(x) |} \leq C_{3},
\end{equation*}
where $C_{3}=C_{1}/c_{1}$.
For given $n \geq 1, x \in \mathcal{I}$ and $g \in \mG$, define
\begin{equation*}
f_{n,x,g}(\eta,t) = g(\eta) \alpha_{n}(x,g)^{T}\psi^{K_{n}}(t), \ (\eta,t) \in \mathcal{E} \times [0,1]^{d}.
\end{equation*}
Consider the class of functions $\mF_{n} = \{ f_{n,x,g} : (x,g) \in \mathcal{I} \times \mG\}$.
We shall apply Corollary \ref{cor: main corollary from theorem} to $\mF_{n}$. Note that $W_{n} = \sup_{f \in \mF_{n}}\bG_{n} f$.
First, we have $| f_{n,x,g}(\eta,t) | \leq C_{3} \xi_{n} | G(\eta) | =: F_{n}(\eta,t)$.
Second, observe that $\mF_{n} = \mathcal{H}_{1} \cdot \mathcal{H}_{2n}$, where $\mathcal{H}_{1}= \{ (\eta,t) \mapsto g(\eta) : g \in \mG \}$ and
$\mathcal{H}_{2n} = \{ (\eta,t) \mapsto \alpha_{n}(x,g)^{T} \psi^{K_{n}}(t) : (x,g) \in \mathcal{I} \times \mG \}$. By condition (C3),
\begin{equation*}
| \alpha_{n}(x,g)^{T} \psi^{K_{n}}(t) - \alpha_{n}(\check{x},\check{g})^{T} \psi^{K_{n}}(t) | \leq L_{n} \xi_{n} \{ | x-\check{x} | + ( \bE[ ( g(\eta_{1}) - \check{g}(\eta_{1}))^{2} )^{1/2} \},
\end{equation*}
so that, using the fact that $\mG$ is VC type, we deduce that there are constants $A,v > 0$ such that
\begin{equation*}
\sup_{Q} N(\mathcal{H}_{2n},e_{Q},\varepsilon C_{3} \xi_{n}) \leq (A L_{n}/\varepsilon)^{v}, \ 0 < \forall \varepsilon \leq 1, \ \forall n \geq 1.
\end{equation*}
Using  again the fact that $\mG$ is VC type and Corollary \ref{covnumber2} (i), we deduce that there are constants $A',v' > 0$ such that
\begin{equation}
\sup_{Q} N(\mF_{n},e_{Q},\varepsilon \| F_{n} \|_{Q,2}) \leq (A' L_{n}/\varepsilon)^{v'}, \ 0 < \forall \varepsilon \leq 1, \ \forall n \geq 1. \label{covnumb3}
\end{equation}
Hence for every $n \geq  1$,
there exists a tight Gaussian random variable $G_{n}$ in $\ell^{\infty}(\mF_{n})$ with mean zero and covariance function
\begin{equation*}
\bE [ G_{n} (f) G_{n}(\check{f}) ] =
\Cov (f(\eta_{1},X_{1}),\check{f}(\eta_{1},X_{1})), \ f,\check{f} \in \mF_{n}.
\end{equation*}
Let $B_{n}(x,g) = G_{n}(f_{n,x,g}), (x,g) \in \mathcal{I} \times \mG$.
Then $B_{n}$ is the desired Gaussian  process.

To apply Corollary \ref{cor: main corollary from theorem}, we make some complimentary calculations. Let $D= \sup_{x \in [0,1]^{d}} \bE[ G^{4}(\eta_{1}) \mid X_{1}=x]$. Then for $n \geq 1$,
\begin{align*}
&\bE [ | g(\eta_{1})  \alpha_{n}(x,g)^{T}\psi^{K_{n}}(X_{1}) |^{3} ] \\
&\quad \leq \bE[ \bE[ G^{3}(\eta_{1}) \mid X_{1} ] | \alpha_{n}(x,g)^{T}\psi^{K_{n}}(X_{1}) |^{3} ] \\
&\quad \leq C_{3} (1+D) \xi_{n} \bE [ | \alpha_{n}(x,g)^{T} \psi^{K_{n}}(X_{1}) |^{2} ] \\
&\quad =C_{3} (1+D) \xi_{n} \alpha_{n}(x,g)^{T} \bE[\psi^{K_{n}}(X_{1}) \psi^{K_{n}}(X_{1})^{T}] \alpha_{n}(x,g) \\
&\quad  \leq  C_{3}^{3} C_{2}(1+D) \xi_{n},  \\
\intertext{and}
&\bE [ | g(\eta_{1}) \alpha_{n}(x,g)^{T}\psi^{K_{n}}(X_{1}) |^{4} ] \leq C_{3}^{4} C_{2} D \xi_{n}^{2}.
\end{align*}
Thus, applying Corollary \ref{cor: main corollary from theorem} with parameters $\gamma,b,\sigma$ in the corollary satisfying $\gamma=\gamma_n=(\log n)^{-1}$, $b=b_{n} = O(\xi_n)$, and $\sigma=O(1)$ shows that there exists a sequence $\widetilde{W}_{n}$ of random variables such that $\widetilde{W}_{n} \stackrel{d}{=} \sup_{f \in \mF_{n}} G_{n}f= \sup_{(x,g) \in \mathcal{I} \times \mG} B_{n}(x,g)$ and as $n \to \infty$,
\begin{equation*}
| W_{n} - \widetilde{W}_{n} | = O_{\bP}(n^{-1/6} \xi_{n}^{1/3} \log n + n^{-1/4} \xi_{n}^{1/2} \log^{5/4} n + n^{-1/2+1/q} \xi_{n} \log^{3/2} n).
\end{equation*}
This completes the proof.
\end{proof}

\newpage

\section{Motivating examples for series empirical processes in Section 3.2}

\label{appendix: motivating examples}

\begin{example}[Forms of $S_n(x,g)$ arising in nonparametric mean regression] 
\label{example: mean} 
Here we explain
which forms of $S_{n}(x,g)$ arise in the nonparametric series or sieve mean regression.
Consider a (generally heteroscedastic)  nonparametric regression model
\begin{equation*}
Y_{i} = m(X_{i}) + \eta_{i}, \ \bE [ \eta_{i} \mid X_{i}] = 0, \ \bE [ \eta_{i}^{2} \mid X_{i} = x] = \sigma^{2} (x), \ 1 \leq i \leq n,
\end{equation*}
where $Y_{i}$ is a scalar response variable, $X_{i}$ is a $d$-vector of covariates of which the support $=[0,1]^{d}$, and $\eta_{i}$ is a scalar unobservable error term.
We assume that the data $(Y_{1},X_{1}),\dots,(Y_{n},X_{n})$ are i.i.d.
The parameter of interest is the conditional mean function $m(x) = \bE[ Y_{1} \mid X_{1} = x]$.

Consider series estimation of $m(x)$. The idea of series estimation is to approximate
$m(x)$ by $\sum_{j=1}^{K_{n}} \theta_{K_{n},j} \psi_{K_{n},j}(x)$ with $K_{n} \to \infty$ as $n \to \infty$ and to estimate the vector $\theta^{K_{n}} = (\theta_{K_{n},1},\dots,\theta_{K_{n},K_{n}})^{T}$ by the least squares method:
\begin{equation*}
\widehat{\theta}^{K_{n}} = \arg \min_{\theta^{K_{n}} \in \R^{K_{n}}} \sum_{i=1}^{n} \left ( Y_{i} - \psi^{K_{n}}(X_{i})^{T}\theta^{K_{n}} \right )^{2}.
\end{equation*}
The resulting estimate of $m(x)$ is given by $\widehat{m}(x) =\psi^{K_{n}}(x)^{T} \widehat{\theta}^{K_{n}}$.

The asymptotic properties of the series estimate have been thoroughly investigated in the literature.
Importantly, under suitable regularity conditions, the rescaled and recentered estimator
 admits an asymptotic linear form:
\begin{equation*}
\tilde S_n(x) = \frac{\sqrt{n}(\widehat{m}(x) - m(x))}{| A_{2n}\psi^{K_{n}}(x)|} \approx  \frac{ \psi^{K_{n}}(x)^{T} A_{1n}}{| A_{2n} \psi^{K_{n}}(x)|} \left [ \frac{1}{\sqrt{n}}\sum_{i=1}^{n} \eta_{i} \psi^{K_{n}}(X_{i}) \right ] =: S_{n}(x),
\end{equation*}
where $A_{1n}=(\bE[  \psi^{K_{n}}(X_{1})\psi^{K_{n}}(X_{1})^{T} ])^{-1}$ and
\[
 A_{2n}= (\bE[ \sigma^{2}(X_{1}) \psi^{K_{n}}(X_{1})\psi^{K_{n}}(X_{1})^{T} ])^{1/2}A_{1n}.
\]
See, for example, \cite{N97}.  Here $\tilde S_{n}(x) \approx S_{n}(x)$ means that $\tilde S_{n}(x) = S_{n}(x) + o_{\bP}(\log^{-1/2} n)$ uniformly in $x \in \mathcal{I}$ (the remainder term could be faster, but $o_{\bP}(\log^{-1/2} n)$ is fast enough to make the remainder term negligible in approximating (in the Kolmogorov distance) the distribution of $\sup_{x \in \mathcal{I}} \tilde{S}_{n}(x)$ by that of the Gaussian analogue of $\sup_{x \in \mathcal{I}} S_{n}(x)$ as the expectation of the latter is typically $O(\sqrt{\log n})$; see Remark \ref{rem: application} and Lemma \ref{lem: anticoncentration}). Hence, for the purpose of making uniform inference on $m(x)$ over a Borel subset $\mathcal{I}$ of $[0,1]^{d}$, it is desirable to have a (tractable) distributional approximation of the quantity $W_{n} = \sup_{x \in \mathcal{I}} S_{n}(x)$.\qed

\end{example}

\begin{example}[Forms of $S_n(x,g)$ arising in nonparametric quantile regression] 
\label{example: quantile}
Here we explain
which forms of $S_{n}(x,g)$ arise in the nonparametric series or sieve quantile regression.
Let $(Y_{1},X_{1}),\dots,(Y_{n},X_{n})$ be i.i.d. random variables taking values in $\R \times \R^{d}$ where the support of $X_{1} = [0,1]^{d}$.
Suppose  that the parameter of interest is the conditional quantile function:
\begin{equation*}
Q(\tau,x)= \inf \{ y : F_{Y|X}(y \mid x)  \geq \tau \}, \ x \in [0,1]^{d}, \tau \in (0,1),
\end{equation*}
where $F_{Y|X} (y \mid x) = \bP( Y_{1} \leq y \mid X_{1} = x)$ is the conditional distribution function.
Consider series estimation of $Q(\tau,x)$. A standard way is to solve the following minimization problem:
\begin{equation*}
\widehat{\theta}^{K_{n}}(\tau) = \arg \min_{\theta^{K_{n}} \in \R^{K_{n}}} \sum_{i=1}^{n} \rho_{\tau} \left ( Y_{i} - \psi^{K_{n}}(X_{i})^{T}\theta^{K_{n}} \right ),
\end{equation*}
where $\rho_{\tau} (y) = \{ \tau - 1(y \leq 0) \} y$ is called the check function \citep{KB78}, and where $K_{n} \to \infty$ as $n \to \infty$. A series estimate of $Q (\tau, x)$ is obtained by $\widehat{Q} (\tau,x) = \psi^{K_{n}}(x)^{T}\widehat{\theta}^{K_{n}}(\tau)$.
Let $\mathcal{T}$ be an arbitrary closed interval in $(0,1)$. Suppose that the conditional distribution function $F_{Y|X} (y \mid x)$ has a Lebesgue density $f_{Y|X}(y \mid x)$.
Then,  subject to some regularity conditions, the rescaled and recentered estimator admits an asymptotically linear form:
\begin{align*}
&\tilde S_n(x,\tau) = \frac{\sqrt{n}(\widehat{Q}(\tau,x) -Q(\tau,x))}{\sqrt{\tau(1-\tau)}| A_{2n}(\tau)\psi^{K_{n}}(x)|} \\
&\approx \frac{ \psi^{K_{n}}(x)^{T}A_{1n}(\tau)}{\sqrt{\tau(1-\tau)}|A_{2n}(\tau)\psi^{K_{n}}(x)|} \left [\frac{1}{\sqrt{n}}\sum_{i=1}^{n} \{ \tau - 1( Y_{i} \leq Q(\tau,X_{i}))\}  \psi^{K_{n}}(X_{i}) \right ] \\
&=: S_{n}(x,\tau),
\end{align*}
where $A_{1n}(\tau) = J_{n}(\tau)^{-1}, J_{n}(\tau)=\bE[ f_{Y|X}( Q(\tau,X_{1}) \mid X_{1}) \psi^{K_{n}}(X_{1})\psi^{K_{n}}(X_{1})^{T}]$, $A_{2n}(\tau) = (\bE[ \psi^{K_{n}}(X_{1})\psi^{K_{n}}(X_{1})^{T} ])^{1/2}J_{n}(\tau)^{-1}$ (note that $\tau (1-\tau)$ comes from the conditional variance of $1(Y_{i} \leq Q(\tau,X_{i}))$ given $X_{i}$). Here too $\tilde S_n(x,\tau) \approx S_n(x,\tau)$ means that $\tilde S_n(x,\tau) = S_n(x,\tau) + o_{\bP}(\log^{-1/2} n)$ uniformly in $(x,\tau) \in \mathcal{I} \times \mathcal{T}$; see \cite{HS00} and \citet[][Theorem 2]{BCF11}. Note that
\begin{equation*}
Y_{i} \leq Q(\tau, X_{i}) \Leftrightarrow \eta_{i} \leq \tau, \ \text{with} \ \eta_{i} = F_{Y|X}(Y_{i} \mid X_{i}),
\end{equation*}
and $\eta_{i}$ are uniform random variables on $(0,1)$, independent of $X_{1},\dots,X_{n}$. So letting $g_{\tau} (\eta) = \tau - 1(\eta \leq \tau)$, we have the expression
\[
S_{n}(x,\tau) = \frac{ \psi^{K_{n}}(x)^{T}A_{1n}(\tau)}{\sqrt{\tau(1-\tau)}|A_{2n}(\tau)\psi^{K_{n}}(x)|} \left [\frac{1}{\sqrt{n}}\sum_{i=1}^{n}g_{\tau} (\eta_{i}) \psi^{K_{n}}(X_{i}) \right ].
\]
For the purpose of making uniform inference on $Q(\tau,x)$ over $(\tau,x) \in \mathcal{T} \times \mathcal{I}$,
it is desirable to have a (tractable) distributional approximation of the quantity $W_{n} = \sup_{(x,\tau) \in \mathcal{I} \times \mathcal{T}} S_{n}(x,\tau)$. \qed
\end{example}

\newpage

\section{Obtaining almost sure bounds from Theorem 2.1}
The purpose of this section is to derive almost sure bounds from Theorem \ref{main}. We use the same notation as that in Section \ref{sec:main}. Consider an infinite sequence $X_1,X_2,\dots$ of i.i.d. random variables taking values in a measurable space $(S,\mathcal{S})$. Let $\mathcal{F}$ be some class of functions defined on $S$. In this section, the function class $\mathcal{F}$ is independent of $n$. For each $n$, denote $Z_n=\sup_{f\in\mathcal{F}}\mathbb{G}_n f$ where
$$
\mathbb{G}_n f = \frac{1}{\sqrt{n}}\sum_{i=1}^n(f(X_i)-\mathbb{E}[f(X_i)]), \ f \in \mF.
$$
We look for conditions under which there exists a sequence of random variables $\widetilde{Z}_n$ such that
$$
\left|Z_n-\widetilde{Z}_n\right|=O_{a.s.}(r_n)
$$
where $r_n\to 0$ as $n\to\infty$ is a sequence of constants and for each $n$, 
\begin{equation}\label{eq: gaussian suprema variables}
\widetilde{Z}_n\stackrel{d}{=} \sup_{f \in \mF} G_{P}f.
\end{equation}
To this end, we have the following theorem:
\begin{theorem}[Almost sure bounds]
Let $\alpha$ and $\beta$ be some constants satisfying $\alpha>1$ and $\beta>q/(q-2)$. Denote $\gamma_n=n^{-1/\beta}(\log n)^{-\alpha}$. Suppose that assumptions (A1), (A2) with $q\geq 3$, and (A4) of Section \ref{sec:main} are satisfied. In addition, suppose that $\kappa=\kappa_n$ and $\varepsilon=\varepsilon_n$ are chosen so that $\kappa_n^3\geq \mathbb{E}[\|\mathbb{E}_n[|f(X_i)|^3]\|_{\mathcal{F}}]$ and $\delta_n(\varepsilon_n,\gamma_n)=O(1)$. Then there exists a sequence $\widetilde{Z}_n$ of random variables satisfying (\ref{eq: gaussian suprema variables}) and
$$
\left|Z_n-\widetilde{Z}_n\right|\leq O_{a.s.}\left(\Delta_n(\varepsilon_n,\gamma_n)+\frac{(\log n)^{\alpha/q}}{n^{1/(2\beta)-1/(q\beta)}}\right)
$$
\end{theorem}
\begin{remark}
One interesting feature of this theorem is that it gives a dimension-free result, that is, the bound does not explicitly depend on the dimensionality $d$ when $S\subset \mathbb{R}^d$.
\end{remark}
\begin{proof}
The proof consists of two steps. In the first step, we construct random variables $\widetilde{Z}_n$ along the subsequence $n=n_m=m^\beta$, $m\geq 1$ so that the conclusion of the theorem holds for this subsequence. In the second step, we show that the conclusion of the theorem holds for all $n$ if we define $\widetilde{Z}_n=\widetilde{Z}_{n_m}$ for all $n_m\leq n< n_{m+1}$.

{\bf Step 1}: Note that by Lemma \ref{lem0}, assumption (A3) is satisfied, and so we can apply Theorem \ref{main} for all $n$. In particular, we can construct $\widetilde{Z}_n$ for all $n=n_m$, $m\geq 1$,  such that
\[
\bP \left  \{  | Z_{n} - \widetilde{Z}_{n} | >  K(q) \Delta_{n} (\varepsilon_{n},\gamma_{n}) \right \}  \leq \gamma_{n} \left \{ 1 + \delta_{n}(\varepsilon_n,\gamma_n) \right \}+\frac{C \log n}{n}
\]
for these $n$. Further, $\delta_n(\varepsilon_n,\gamma_n)=O(1)$, $\alpha>1$, and $\beta>q/(q-2)>1$ imply that
$$
\sum_{m=1}^\infty \left[\gamma_{n_m} \left \{ 1 + \delta_{n_m}(\varepsilon_{n_m},\gamma_{n_m}) \right \}+\frac{C \log n_m}{n_m}\right]< \infty,
$$
so that it follows from the Borel-Cantelli lemma that
\begin{equation}\label{eq: result for subsequence}
\left|Z_{n_m}-\widetilde{Z}_{n_m}\right|\leq O_{a.s.}\left(\Delta_{n_m}(\varepsilon_{n_m},\gamma_{n_m})\right)\text{ as $m\to\infty$}.
\end{equation}
This completes Step 1.

{\bf Step 2:} Let $s_m$, $m\geq 1$, be some sequence of constants to be chosen later.
We will use the Montgomery-Smith maximal inequality \citep[see][]{M93,DG99}:
\begin{equation}\label{eq: MS inequality}
\mathbb{P}\left\{\max_{1\leq k\leq n}\|\sqrt{k}\mathbb{G}_k\|_{\mathcal{F}}>30s\right\}\leq 9\mathbb{P}\left\{\|\sqrt{n}\mathbb{G}_n\|_{\mathcal{F}}>s\right\}\text{ for all $s>0$}.
\end{equation}
Using (\ref{eq: MS inequality}) and setting $\Delta n_m=n_{m+1}-n_m$, we obtain for all $m\geq 1$,
\begin{align}
&\mathbb{P}\left\{\max_{n_m< n<n_{m+1}}\left \| \sum_{i=n_{m}+1}^n (f(X_i)-\mathbb{E}[f(X_i)]) \right \|_{\mF} >30 s_m\right\}\nonumber\\
&\qquad\leq \mathbb{P}\left\{\max_{1\leq k\leq \Delta n_m}\|\sqrt{k}\mathbb{G}_k\|_{\mathcal{F}}>30s_m\right\}\nonumber\\
&\qquad\leq 9\mathbb{P}\left\{\|(\Delta n_m)^{1/2}\mathbb{G}_{\Delta n_m}\|_{\mathcal{F}}>s_m\right\}\label{eq: prob to be bounded}
\end{align}
Further, setting $t_m=(m(\log m)^{\alpha})^{2/q}$ and
\begin{multline*}
s_m=(\Delta n_m)^{1/2}\Big\{(1+\alpha)\mathbb{E}[\|\mathbb{G}_{\Delta n_m}\|_{\mathcal{F}}]+K(q)\Big [(\|F\|_{P,2}+\\(\Delta n_m)^{-1/2+1/q}\|F\|_{P,q})\sqrt{t_m}
+\alpha^{-1}(\Delta n_m)^{-1/2+1/q}\|F\|_{P,q}t_m\Big ]\Big\}
\end{multline*}
where $K(q)$ is a sufficiently large constant, we obtain from Theorem \ref{concentration} that the probability in (\ref{eq: prob to be bounded}) is bounded from above by $t_m^{-q/2}$. Our choice of $t_m$ ensures that
\[
\sum_{m=1}^\infty t_m^{-q/2}<\infty,
\]
so that applying the Borel-Cantelli lemma one more time, we obtain
\begin{equation}\label{eq: discretization result}
\max_{n_m< n<n_{m+1}}\left \| \sum_{i=n_{m}+1}^n (f(X_i)-\mathbb{E}[f(X_i)]) \right \|_{\mF} =O_{a.s.}(s_m)\text{ as }m\to\infty.
\end{equation}
Note also that $\Delta n_m\leq \beta m^{\beta-1}$, and Theorem \ref{vdVW} implies that $\mathbb{E}[\|\mathbb{G}_{\Delta n_m}\|_{\mathcal{F}}]=O(1)$, so that
$$
s_m=O((\Delta n_m)^{1/2}\sqrt{t_m})
$$
since $\beta>q/(q-2)$, which we assume. Substituting $\Delta n_m$ and $t_m$ gives
\begin{equation}\label{eq: simplification result}
s_m=O(m^{(\beta-1)/2+1/q}(\log m)^{\alpha/q})\text{ as }m\to\infty
\end{equation}
Combining (\ref{eq: result for subsequence}), (\ref{eq: discretization result}), and (\ref{eq: simplification result}) together with defining $\widetilde{n} := \widetilde{n}_{n} :=n_m$ for all $n_m\leq n< n_{m+1}$ and $\widetilde{Z}_{n} := \widetilde{Z}_{\widetilde{n}} \stackrel{d}{=} \sup_{f \in \mF} G_{P}f$ for all $n_m<n<n_{m+1}$, we have
\begin{equation}\label{eq: penultimate approximation}
\left|Z_n-(\widetilde{n}/n)^{1/2}\widetilde{Z}_n \right|\leq O_{a.s.}\left(\Delta_n(\varepsilon_n,\gamma_n)+\frac{(\log n)^{\alpha/q}}{n^{1/(2\beta)-1/(q\beta)}}\right).
\end{equation}
It remains to bound $|(\widetilde{n}/n)^{1/2}\widetilde{Z}_n-\widetilde{Z}_n|$. To this end, note that $\widetilde{Z}_n$ is the supremum of a zero-mean Gaussian process, whose distribution is independent of $n$. Moreover, $\widetilde{Z}_n$ is finite almost surely. Therefore, it follows from Proposition A.2.3 in \cite{VW96} that there exists a constant $K'$ such that 
\[
\mathbb{E}[\exp \{ K'(\widetilde{Z}_n)^2 \}]=O(1).
\]
Therefore, $\widetilde{Z}_n=O_{a.s.}(\sqrt{\log n})$. Since $(\widetilde{n}/n)^{1/2}-1=O(n^{-1/\beta})$, we conclude that
\begin{equation}\label{eq: approximation of gaussian processes}
\left|(\widetilde{n}/n)^{1/2}\widetilde{Z}_n-\widetilde{Z}_n\right|=O_{a.s.}\left(\frac{\sqrt{\log n}}{n^{1/\beta}}\right).
\end{equation}
Combining (\ref{eq: penultimate approximation}) and (\ref{eq: approximation of gaussian processes}) completes the proof.
\end{proof}

\end{document}